\documentclass[10pt]{amsart}

\usepackage[colorlinks]{hyperref}
\usepackage{color,graphicx,shortvrb}
\usepackage[latin 1]{inputenc}

\newtheorem{theorem}{Theorem}[section]

\newtheorem{corollary}[theorem]{Corollary}

\newtheorem{lemma}[theorem]{Lemma}

\newtheorem{proposition}[theorem]{Proposition}

\usepackage[active]{srcltx} %SRC Specials for DVI search
\usepackage{stackrel}
\usepackage{amssymb}

\def\J#1#2#3{ \left\{ #1,#2,#3 \right\} }

\def\NN{{\mathbb{N}}}
\def\11{\textbf{$1$}}
\def\CC{{\mathbb{C}}}

\newcommand{\w}{\omega}

\usepackage{enumerate}
\usepackage{amsmath}
\usepackage{amssymb}
\usepackage{mathabx}

\begin{document}

\title[One-parameter groups of orthogonality preservers on JB$^*$-algebras]{One-parameter groups of orthogonality preservers on JB$^*$-algebras}

\author[J.J. Garc{\' e}s]{Jorge J. Garc{\' e}s}
\email{j.garces@upm.es}
\address{Departamento de Matem{\' a}tica Aplicada a la Ingenier{\' i}a Industrial, ETSIDI, Universidad Polit{\' e}cnica de Madrid, Madrid, Spain.}

\author[A.M. Peralta]{Antonio M. Peralta}
\email{aperalta@ugr.es}
\address{Departamento de An{\'a}lisis Matem{\'a}tico, Facultad de
Ciencias, Universidad de Granada, 18071 Granada, Spain.}



\date{}

\maketitle

\begin{abstract} In a first objective we improve our understanding about surjective and bijective bounded linear operators preserving orthogonality from a JB$^*$-algebra $\mathcal{A}$ into a JB$^*$-triple $E$. Among many other conclusions, it is shown that a bounded linear bijection $T: \mathcal{A}\to E$ is orthogonality preserving if, and only if, it is biorthogonality preserving if, and only if, it preserves zero-triple-products in both directions (i.e., $\{a,b,c\}=0 \Leftrightarrow \{T(a),T(b),T(c)\}=0$). In the second main result we establish a complete characterization of all one-parameter groups of orthogonality preserving operators on a JB$^*$-algebra.
\end{abstract}

\section{Introduction}\label{intro}

Let $T:A\to B$ be a mapping between C$^*$-algebras. We say that $T$ is \emph{orthogonality preserving} if it maps orthogonal elements in $A$ to orthogonal elements in $B$. Along this paper, elements $a,b$ in a C$^*$-algebra $A$ are called orthogonal if $a b^* = b^* a= 0$. If $T(a^*) = T(a)^*$ for all $a\in A$, the mapping $T$ is called symmetric. Bounded linear orthogonality preserving operators between C$^*$-algebras were fully determined by M. Burgos, F.J. Fern{\'a}ndez-Polo, J. Mart{\'i}nez and the authors of this note in \cite{BurFerGarMarPe2008}, previously M. Wolff \cite{Wolff94} had determined the precise form of all symmetric bounded linear orthogonality preserving operators between unital C$^*$-algebras. The just quoted reference also contains a detailed study of uniformly continuous one-parameter groups of one-parameter groups of symmetric orthogonality preserving operators on unital C$^*$-algebras.\smallskip

Let us recall that \emph{one-parameter semigroup} of bounded linear operators on a Banach space $X$ is a correspondence $\mathbb{R}_0^{+} \to B(X),$ $t\mapsto T_t$ satisfying $T_{t+s} = T_{s} T_{t}$ for all $s,t\in \mathbb{R}_0^{+}$ and $T_0 =I$, where $B(X)$ stands for the Banach space of all bounded liner operators on $X$. A one-parameter semigroup $\{T_t: t\in\mathbb{R}_0^+ \}$ is uniformly continuous at the origin, i.e. $\displaystyle \lim_{t\to 0} \|T_t -I\| =0$, if and only if there exists a bounded linear operator $R\in B(X)$ such that $T_t =e^{t R}$ for all $t\in \mathbb{R}_0^+$, and in such a case, $T_t$ extends to a uniformly continuous one-parameter group on $\mathbb{R}$ (compare \cite[Proposition 3.1.1]{BratRob1987}).\smallskip

Theorem 2.6 in \cite{Wolff94} asserts that if $\{T_t: t\in \mathbb{R}_0^{+}\}$ is a uniformly continuous one-parameter semigroup of orthogonality preserving symmetric operators on a unital C$^*$-algebra $A$, then there exists a uniquely determined element $h$ in the center of $A$, and also a uniquely determined uniformly continuous group $\{S_t: t\in \mathbb{R}_0^{+}\}$ of $^*$-automorphisms on $A$ such that $T_t (a) = e^{th} S_t(a)$ for all $a\in A$, $t\in \mathbb{R}_0^+$.\smallskip

In a recent note we determined all uniformly continuous one-parameter semigroups of orthogonality preserving operators on general C$^*$-algebras (see \cite{GarPeUnitCstaralg20}). In the general case the conclusion is technically more complex. Namely, suppose $\{T_t: t\in \mathbb{R}_0^{+}\}$ is a uniformly continuous one-parameter semigroup of orthogonality preserving operators on a C$^*$-algebra $A$, then there exists a uniformly continuous one-parameter semigroup $\{S_t: t\in \mathbb{R}_0^{+}\}$ of surjective linear isometries {\rm(}i.e. triple isomorphisms{\rm)} on $A$ such that the identities $$ h_{t+s} =  h_t r_t^* S_t^{**} (h_s)=T^{**}_t (h_s),\  h_t^* S_t(x)= S_t(x^*)^* h_t, \ h_t S_t(x^*)^* = S_t(x) h_t^*,$$ $$h_t r_t^* S_t(x) = S_t(x) r_t^* h_t, \hbox{ and } T_t(x) = h_t r_t^* S_t(x) = S_t(x) r_t^* h_t,$$  hold for all $s,t\in \mathbb{R},$ $x\in A$, where $h_t= T_t^{**} (1)$ and $r_t$ is the range partial isometry of $h$ in $A^{**}$ (cf. \cite[Theorem 3]{GarPeUnitCstaralg20}). Among the tools employed in this result we decode the structure of all surjective (respectively, bijective) orthogonality preserving bounded linear operators between C$^*$-algebras. In the general case the sets $\{r_t:t\in \mathbb{R}\}$ and  $\{h_t:t\in \mathbb{R}\}$ need not be one-parameter groups. \smallskip

C$^*$-algebras are englobed in the wider class of JB$^*$-algebras, and all these structures are inside the class of complex Banach spaces known as JB$^*$-triples, where the notion of orthogonality makes also sense (see Subsection \ref{subsec:defi} for a detailed review of these notions). Actually, as shown in \cite{BurFerGarMarPe2008} and \cite{BurFerGarPe09}, the setting and terminology of JB$^*$-triples seem to be the appropriate language to describe continuous orthogonality preserving linear maps between C$^*$- and JB$^*$-algebras. The main result in \cite{BurFerGarPe09} describes those continuous linear maps preserving orthogonality from a JB$^*$-algebra into a JB$^*$-triple. The present note contains one of the first studies on uniformly continuous one-parameter semigroups of orthogonality preserving operators on JB$^*$-algebras. Our main goal is to extend to the Jordan setting the conclusions we recently obtained in the case of uniformly continuous one-parameter semigroups of orthogonality preserving operators on general C$^*$-algebras commented above.\smallskip

As it was shown in \cite{BurFerGarMarPe2008,BurFerGarPe09}, the terminology of JB$^*$-triples seems to be the optimal language to understand continuous linear orthogonality preservers between C$^*$-algebras and, more generally, from a JB$^*$-algebra into a JB$^*$-triple. All bounded linear operators preserving orthogonality from a JB$^*$-algebra into a JB$^*$-triple were completely determined in terms of Jordan and triple homomorphisms, up to multiplication by an element in the second dual of the codomain satisfying certain commutativity identities (cf. Theorem \ref{thm characterization of OP JBstar} below). The recent paper \cite{GarPeUnitCstaralg20} shows that in the case of bounded linear orthogonality preserving operators between C$^*$-algebras those being surjective or bijective enjoy additional properties. In Section \ref{sec: surjective OP JBstar} we study surjective and bijective continuous orthogonality preserving maps from a JB$^*$-algebra  into a JB$^*$-triple $E$. We prove that for each bijective bounded linear operator $T$ from a JB$^*$-algebra $\mathcal{A}$ to a JB$^*$-triple $E$, if we set $h=T^{**} (1)\in E^{**}$ and $r$ denotes the range tripotent of $h$ in $E^{**}$. Then the following statements are equivalent:\begin{enumerate}[$(a)$] \item $T$ is orthogonality preserving;
\item The elements $h$ and $r$ belong to $M(E)$ with $h$ invertible and $r$ unitary. Consequently, $E^{**}$ is a JBW$^*$-algebra with respect to the Jordan product and involution defined by $a\circ_r b = \{a,r,b\}$ and $a^{*_r} = \{r,a,r\}$, respectively, and contains $E$ as a JB$^*$-subalgebra. Furthermore, there exists a Jordan $^*$-isomorphism $S: \mathcal{A}\to (E,\circ_{r},*_{r})$ such that $h$ lies in $Z(E^{**},\circ_{r},*_{r})$, and $T(x)=h\circ_{r(h)} S(x) = U_{h^{\frac12}} S(x),$ for every $x\in \mathcal{A}$, where $h^{\frac12}$ denotes the square root of the positive element $h$ in the JB$^*$-algebra $E_2^{**} (r)$;
\item $T$ is biorthogonality preserving;
\item $T$ is orthogonality preserving on $\mathcal{A}_{sa}$;
\item $T$ is orthogonality preserving on $\mathcal{A}^{+}$;
\item $T$ preserves zero-triple-products, i.e. $$\{a,b,c\}=0 \Rightarrow \{T(a),T(b),T(c)\}=0;$$
\item $T$ preserves zero-triple-products in both directions, i.e. $$\{a,b,c\}=0 \Leftrightarrow \{T(a),T(b),T(c)\}=0,$$
\end{enumerate} (see Propositions \ref{p surjective OP preserves multipliers JB} and \ref{p Ortpreserving on Asa JBstar} and Corollary \ref{c Characterization bd OP plus bijective Jordan}). The proofs in the Jordan setting require a complete new argument which is not a mere aesthetical adaptation of those valid for maps between C$^*$-algebras.\smallskip

Section \ref{sec: one-parametri groups OP Jordan} is devoted to the study of one-parameter groups of orthogonality preserving operators on a JB$^*$-algebras. We shall see in this note how the terminology and tools from JB$^*$-triple theory makes more accesible (and perhaps more complete) the study of one-parameter groups.  Firstly, uniformly continuous one-parameter groups of surjective isometries on a general JB$^*$-triple are precisely given by the exponential of real multiples of triple derivations (see Lemma \ref{l -parameter group of iso on a JB*-triple}). The paper culminates with the following characterization of one-parameter groups of orthogonality preserving operators on a general JB$^*$-algebra $\mathcal{A}$: Suppose $\{T_t: t\in \mathbb{R}_0^{+}\}$ is a family of orthogonality preserving bounded linear bijections on $\mathcal{A}$ with $T_0=Id$. For each $t\geq 0$ let $h_t = T_t^{**} (1)$ and let $r_t$ be the range tripotent of $h_t$ in $\mathcal{A}^{**}$. According to what we prove in Section \ref{sec: surjective OP JBstar}, $h_t$ and $r_t$ belong to $M(\mathcal{A})$ with $h_t$ invertible and $r_t$ unitary and there exists a Jordan $^*$-isomorphism $S_t: \mathcal{A} \to (\mathcal{A},\circ_{r_t},*_{r_t})$ such that $h_t$ lies in $Z(\mathcal{A}^{**},\circ_{r_t},*_{r_t})$, and $T_t(x)=h_t\circ_{r_t} S_t(x) = U_{h_t^{\frac12}} S_t(x),$ for every $x\in \mathcal{A}$. We shall show that under these hypotheses the following statements are equivalent:\begin{enumerate}[$(a)$]\item $\{T_t: t\in \mathbb{R}_0^{+}\}$ is a uniformly continuous one-parameter semigroup of orthogonality preserving operators on $\mathcal{A}$;
\item $\{S_t: t\in \mathbb{R}_0^{+}\}$ is a uniformly continuous one-parameter semigroup of surjective linear isometries {\rm(}i.e. triple isomorphisms{\rm)} on $\mathcal{A}$ {\rm(}and hence there exists a triple derivation $\delta$ on $\mathcal{A}$ such that $S_t = e^{t \delta}$ for all $t\in \mathbb{R}${\rm)}, the mapping $t\mapsto h_t $ is continuous at zero, and the identity \begin{equation}\label{eq new idenity in the statement of theorem 1 on one-parameter Jordan} h_{t+s} =  h_t \circ_{r_t} S_t^{**} (h_s)= \{ h_t , {r_t},  S_t^{**} (h_s) \},
    \end{equation} holds for all $s,t\in \mathbb{R}$
\end{enumerate} (see Theorem \ref{t Wolff one-parameter for OP JBstar}).

\subsection{JB$^*$-algebras and JB$^*$-triples}\label{subsec:defi}

A real or complex Jordan algebra is a non-necessarily associative algebra $\mathcal{B}$ over $\mathbb{R}$ or $\mathbb{C}$ whose product (denoted by $\circ$) is commutative and satisfies the so-called \emph{Jordan identity}: $( x \circ y ) \circ x^2 = x\circ ( y\circ x^2 ).$ Given an element $a \in \mathcal{B}$ the symbol $U_a$ will stand for the linear mapping on $\mathcal{B}$ defined by $U_a (b) := 2(a\circ b)\circ a - a^2\circ b$. A Jordan Banach algebra $\mathcal{B}$ is a Jordan algebra equipped with a complete norm satisfying $\|a\circ b\|\leq \|a\| \cdot \|b\|$ for all $a,b\in \mathcal{B}$. A JB$^*$-algebra is a complex Jordan Banach algebra $\mathcal{B}$ equipped with an algebra involution $^*$ satisfying  $\|U_a ({a^*}) \|= \|a\|^3$, $a\in \mathcal{B}$. %, where $U_a(b)  = 2 (a\circ b) \circ a - a^2 \circ b$ for all $b\in \mathcal{B}$.
A \emph{JB-algebra} is a real Jordan algebra $J$ equipped with a complete norm satisfying \begin{equation}\label{eq axioms of JB-algebras} \|a^{2}\|=\|a\|^{2}, \hbox{ and } \|a^{2}\|\leq \|a^{2}+b^{2}\|\ \hbox{ for all } a,b\in J.
\end{equation} A celebrated result due to J.D.M. Wright shows that every JB-algebra $J$ corresponds uniquely to the self-adjoint part $\mathcal{B}_{sa}=\{x\in B : x^* =x\}$ of a JB$^*$-algebra $\mathcal{B}$ \cite{Wright77}.\smallskip

Every C$^*$-algebra is a JB$^*$-algebra with respect to its original norm and involution and the natural Jordan product given by $a \circ b := \frac12 (a b + ba)$. A JB$^*$-algebra is said to be a \emph{JC$^*$-algebra} if it is a JB$^*$-subalgebra of some C$^*$-algebra.\smallskip

As seen in many previous references, the notion of orthogonality and the study of orthogonality preservers is easier when these structures are regarded as elements in the class of JB$^*$-triples. There are undoubted motivations coming from holomorphic theory on general complex Banach spaces to be attracted by the notion of JB$^*$-triple (see, for example, \cite{Harris74,Ka}). As introduced by W. Kaup in \cite{Ka} a JB$^*$-triple is a complex Banach space $E$ admitting a continuous triple product $\J ... : E\times
E\times E \to E,$ which is conjugate linear in the central variable and symmetric and bilinear in the outer variables satisfying:
\begin{enumerate}[{\rm (a)}] \item (Jordan identity) $$L(a,b) L(x,y) = L(x,y) L(a,b)
+ L(L(a,b)x,y)  - L(x,L(b,a)y),$$ for all $a,b,x,y\in E$, where $L(a,b)$ is the linear operator on
$E$ defined by $L(a,b) x = \J abx;$
\item For each $a\in E$, $L(a,a)$ is a hermitian operator with non-negative spectrum;
\item $\|\{a,a,a\}\| = \|a\|^3$ for all $a\in E$.
\end{enumerate}

The triple product \begin{equation}\label{eq Cstar triple product} \J xyz := \frac{1}2 (xy^*z + zy^*x),\ \ \ \ \ \ \ \  (x,y,z\in A).
\end{equation} can be indistinctly applied to induce a structure of JB$^*$-triple on every C$^*$-algebra and on the space $B(H,K)$ of all bounded linear operators between two complex Hilbert spaces $H$ and $K$ (see \cite[Fact 4.1.41]{Cabrera-Rodriguez-vol1}). Furthermore, each JB$^*$-algebra is a JB$^*$-triple with the same norm and triple product \begin{equation}\label{eq triple product JBstar algebra} \J xyz = (x\circ y^*) \circ z + (z\circ y^*)\circ x - (x\circ z)\circ y^*\end{equation} (see \cite[Theorem 4.1.45]{Cabrera-Rodriguez-vol1}). By a JB$^*$-subtriple of a JB$^*$-triple $E$ we mean a norm closed subspace of $E$ which is also closed for the triple product.\smallskip

A JBW$^*$-triple is a JB$^*$-triple which is also a dual Banach space. The bidual of every JB$^*$-triple is a JBW$^*$-triple \cite{Di86}. The alter ego of Sakai's theorem asserts that every JBW$^*$-triple admits a unique (isometric) predual and its product is separately weak$^*$ continuous \cite{BaTi} (see also \cite[Theorems 5.7.20 and 5.7.38]{Cabrera-Rodriguez-vol2}).\smallskip

A \emph{triple homomorphism} between JB$^*$-triples $E$ and $F$ is a linear mapping $T:E\to F$ satisfying $T\{a,b,c\} = \{T(a),T(b), T(c)\}$, for all $a,b,c\in E$.\smallskip

An element $e$ in a JB$^*$-triple $E$ is called a \emph{tripotent} if $\{e,e,e\} =e$. Each tripotent $e\in E$ induces a \emph{Peirce decomposition} of $E$ in the form $$E= E_{2} (e) \oplus E_{1} (e) \oplus E_0 (e),$$ where for
$j=0,1,2,$ $E_j (e)$ is the $\frac{j}{2}$ eigenspace of the operator $L(e,e)$.
%The Peirce decomposition satisfies certain rules known as \emph{Peirce arithmetic}: $$\J {E_{i}(e)}{E_{j} (e)}{E_{k} (e)}\subseteq E_{i-j+k} (e),$$ if $i-j+k \in \{ 0,1,2\}$ and is zero otherwise. In addition, $$\J {E_{2} (e)}{E_{0}(e)}{E} = \J {E_{0} (e)}{E_{2}(e)}{E} =0.$$
For $j\in \{0,1,2\}$ the corresponding \emph{Peirce $j$-projection} of $E$ onto $E_j (e)$ is denoted by $P_{j} (e)$. The Peirce 2-subspace $E_2 (e)$ is a JB$^*$-algebra with Jordan product $x\circ_e y := \J xey$ and involution $x^{*_e} := \J
exe$ (cf. \cite[\S 4.2.2, Fact 4.2.14 and Corollary 4.2.30]{Cabrera-Rodriguez-vol1}). The reader should be warned that in \cite{Cabrera-Rodriguez-vol1} the Peirce subspaces $E_0(e)$, $E_1(e)$ and $E_2(e)$ are denoted by $E_0(e)$, $E_{\frac12}(e)$ and $E_1(e)$, respectively.\smallskip

Suppose $0\neq a$ is an element in a JB$^*$-triple $E$. We set $a^{[1]}= a$, $a^{[3]} = \J aaa$, and $a^{[2n+1]} := \J aa{a^{[2n-1]}},$ $(n\in \NN)$. The norm closure of the linear span of the odd powers $a^{[2n+1]}$ defines a JB$^*$-subtriple of $E$ which coincides with the JB$^*$-subtriple generated by $a,$ and will be denoted by $E_a$. There exists a triple Gelfand theory for single generated JB$^*$-subtriples assuring that $E_a$ is JB$^*$-triple isomorphic (and hence isometric) to some $C_0 (\Omega_{a})$ for some (unique) compact Hausdorff space $\Omega_{a}$ contained in the set $[0,\|a\|],$ such that $0$ cannot be an isolated point in $\Omega_a$. Here and throughout the paper, the symbol $C_0 (\Omega_{a})$ will stand for the Banach space of all complex-valued continuous functions on $\Omega_a$ vanishing at $0.$ We can further find a triple isomorphism $\Psi : E_a \to C_{0}(\Omega_a)$ satisfying $\Psi (a) (t) = t$ $(t\in \Omega_a)$ (cf. \cite[Corollary 4.8]{Ka0}, \cite[Corollary 1.15]{Ka} and \cite{FriRu85} or \cite[Lemma 3.2, Corollary 3.4 and Proposition 3.5]{Ka96}). The set $\Omega_a$ is called \emph{the triple spectrum} of $a$ (in $E$), and it does not change when computed with respect to any JB$^*$-subtriple $F\subseteq E$ containing $a$ \cite[Proposition 3.5]{Ka96}. As in \cite{Ka96} for $a=0$ we set $\Omega_a =\emptyset$.\smallskip

Let us recall some properties of the \emph{triple functional calculus}. Given an element $a$ in JB$^*$-triple $E$, let $\Omega_a$ and $\Psi : E_a \to C_{0}(\Omega_a)$ denote the triple spectrum of $a$ and the triple isomorphism given in previous paragraphs. For each continuous function $f\in C_{0}(\Omega_a)$, $f_t(a)$ will denote the unique element in $E_a$ such that $\Psi (f_t(a)) = f$. The element $f_t(a)$ will be called the \emph{continuous triple functional calculus} of $f$ at $a$. For example, the function $g(t) =\sqrt[3]{t}$ produces $g_t(a) = a^{[\frac{1}{3}]}$. %To avoid confusions, if $a$ is a hermitian (or more generally a normal) element in a C$^*$-algebra $A$, and $f$ is a continuous function on the spectrum $\sigma(a)$ of $a$ vanishing at zero, the symbol $f(a)$ will stand for the usual continuous functional calculus in the C$^*$-algebra $A$. \smallskip
\smallskip

Among the consequences of this local triple Gelfand theory, for each $a$ in a JB$^*$-triple $E$, there exists a unique \emph{cubic root} of $a$ in $E_a$, i.e. and element $a^{[\frac13 ]}\in E_a$ satisfying \begin{equation}\label{eq existence of cubic roots} \J {a^{[\frac13 ]}}{a^{[\frac13 ]}}{a^{[\frac13 ]}}=a.
\end{equation} The sequence $(a^{[\frac{1}{3^n}]})_n$ can be recursively defined by $a^{[\frac{1}{3^{n+1}}]} = \left(a^{[\frac{1}{3^{n}}]}\right)^{[\frac 13]}$, $n\in \NN$. The sequence $(a^{[\frac{1}{3^n}]})_n$ converges in the weak$^*$-topology of $E^{**}$ to a (unique) tripotent denoted by $r(a)$. The tripotent $r(a)$ is named the \emph{range tripotent} of $a$ in $E^{**}$. Alternatively, $r(a)$ is the smallest tripotent $e\in E^{**}$ satisfying that $a$ is positive in the JBW$^*$-algebra $E^{**}_{2} (e)$ (compare \cite[Lemma 3.3]{EdRu96}).\smallskip

The relation ``being orthogonal'' extends naturally from the C$^*$-setting to the setting of JB$^*$-triples, and thus to JB$^*$-algebras. Elements $a,b$ in a JB$^*$-triple $E$ are \emph{orthogonal} (denoted by $a \perp b$) if $L(a,b) = 0$. Lemma 1 in \cite{BurFerGarMarPe2008} contains several reformulations of the relation ``being orthogonal'' which will be applied without any explicit mention. It should be also noted that elements $a,b$ in a C$^*$-algebra $A$ are orthogonal in the C$^*$- sense if and only if they are orthogonal when $A$ is regarded as a JB$^*$-triple.\smallskip

Let $\mathcal{B}$ be a JB$^*$-algebra and $I\subseteq \mathcal{B}$ a (norm closed) subspace of $\mathcal{B}$. We shall say that $I$ is a (closed Jordan) ideal of $\mathcal{B}$ if $I\circ \mathcal{B} \subseteq I$. A (closed) \emph{triple ideal} or simply an \emph{ideal} of a JB$^*$-triple $E$ is a (norm closed) subspace $I\subseteq E$ satisfying $\{E,E,I\}+\{E,I,E\}\subseteq I$, equivalently, $\{E,E,I\}\subseteq I$ or $\{E,I,E\}\subseteq I$ or $\{E,I,I\}\subseteq I$ (see \cite[Proposition 1.3]{BuChu92}).\smallskip

We refer to \cite{Cabrera-Rodriguez-vol1,HOS} for the basic background on JB$^*$-triples and JB$^*$-algebras.

\section{Orthogonality preserving linear surjections from a JB$^*$-algebra}\label{sec: surjective OP JBstar}

A (linear) mapping $T$ between JB$^*$-triples is \emph{orthogonality preserving} (respectively, \emph{biorthogonality preserving}) if $T(a) \perp T(b)$ for every $a\perp b$ in the domain (respectively, $T(a)\perp T(b) \Leftrightarrow a\perp b$).\smallskip

For each element $a$ in a JB$^*$-algebra $\mathcal{B}$ the (Jordan) multiplication operator by $a$, $M_a,$ is defined by $M_a (x) = a\circ x$ ($x\in \mathcal{B}$). Elements $a$ and $b$ in $\mathcal{B}$ are said to \emph{operator commute} in $\mathcal{B}$ if the multiplication operators $M_a$ and $M_b$ commute in $B(\mathcal{B})$, i.e., $$(a\circ x) \circ b = a\circ (x\circ b), \hbox{ for all $x$ in $\mathcal{B}$.}$$ The center of $\mathcal{B}$, $Z(\mathcal{B})$, is the set of all elements $z$ in $\mathcal{B}$ such that $z$ and $b$ operator commute for every $b$ in $\mathcal{B}$.\smallskip

Let us assume that $h$ and $k$ are two self-adjoint elements in a JB$^*$-algebra $\mathcal{B}$. It is known that $h$ and $k$ generate a JB$^*$-subalgebra of $\mathcal{B}$ which can be realised as a JC$^*$-subalgebra of some $B(H)$ (cf. \cite[Corollary 2.2]{Wright77} and Macdonald's and Shirshov-Cohn's theorems \cite[Theorems 2.4.13 and 2.4.14]{HOS}), and under this identification, $h$ and $k$ commute as element of $B(H)$ in the usual sense whenever they operator commute in $\mathcal{B}$ (cf. \cite[Proposition 1]{Topping}).  Similarly, it can be shown that $h$ and $k$ operator commute in $\mathcal{B}$ if and only if the following identity holds \begin{equation}\label{eq ideintity for operator commutativity of two hermitian elements}\hbox{ $h^2 \circ k =\J hkh$ (equivalently, $h^2 \circ k = 2 (h\circ k)\circ h - h^2 \circ k$).}
\end{equation}

It is known that each closed Jordan ideal $I$ of a JB$^*$-algebra $\mathcal{B}$ is self-adjoint --i.e. $I^* = I$-- (see \cite[Theorem 17]{Young78} or \cite[Proposition 3.4.13]{Cabrera-Rodriguez-vol1}). Therefore, every Jordan ideal of $\mathcal{B}$ is a triple ideal. A norm closed subspace $I$ of a JB$^*$-algebra $\mathcal{B}$ is a triple/Jordan ideal if and only if it is the same kind of substructure in the unitization of $\mathcal{B}$. So we can assume that $\mathcal{B}$ is unital. If $I$ is a triple ideal of $\mathcal{B}$ the identities $\{ A, 1, I \}  =  A \circ I$   and $\{ 1, I, 1\}  =  \{ x^*  :  x  \in  I \}= I^*$ show that $I$ is a self-adjoint Jordan ideal of $\mathcal{B}$.\label{page Jordan and triple ideals coincide}\smallskip

Let $\mathcal{B}$ be a JB$^*$-algebra. The \emph{(Jordan) multipliers algebra} of $\mathcal{B}$ is defined as $$M_{J}( \mathcal{B}):=\{x\in \mathcal{B}^{**}: x\circ  \mathcal{B} \subseteq  \mathcal{B}\}$$ (cf. \cite{Ed80}). The space $M_{J}(\mathcal{B})$ is a unital JB$^*$-subalgebra of $\mathcal{B}^{**}$. Moreover, $M_{J}(\mathcal{B})$ is the (Jordan) idealizer of $\mathcal{B}$ in $\mathcal{B}^{**}$, that is, the largest JB$^*$-subalgebra of $\mathcal{B}^{**}$ which contains $\mathcal{B}$ as a closed Jordan ideal. If we realize $ \mathcal{B}$ as a JB$^*$-triple we can also consider its \emph{triple multipliers} in $\mathcal{B}^{**}$ given by $$M (\mathcal{B}):=\{ x\in \mathcal{B}^{**}: \{x, \mathcal{B}, \mathcal{B}\}\subseteq  \mathcal{B}\}$$ as defined in \cite{BuChu92}. In this case $M(\mathcal{B})$ is the largest JB$^*$-subtriple of $\mathcal{B}^{**}$ which contains $\mathcal{B}$ as a closed triple ideal (see \cite[Theorem 2.1]{BuChu92}). Let us observe that these two notions do not cause any source of conflict because the triple multipliers of $\mathcal{B}$ clearly is a JB$^*$-subalgebra of $\mathcal{B}^{**}$ since it contains the unit of $\mathcal{B}^{**},$ and since Jordan and triple ideals of a JB$^*$-algebra are the same (cf. the comments in page \pageref{page Jordan and triple ideals coincide}) both notions above coincide.\smallskip

The determination of all bounded linear operators preserving orthogonality between JB$^*$-triples remains as an open problem. However, as shown in \cite{BurFerGarPe09}, the description is affordable if the domain is a JB$^*$-algebra.\smallskip

\begin{theorem}\label{thm characterization of OP JBstar}\cite[Theorem 4.1 and Corollary 4.2]{BurFerGarPe09} Let $T:\mathcal{A}\to E$ be a bounded linear operator from a JB$^*$-algebra to a JB$^*$-triple, let $h=T^{**}(1)$ and let $r=r(h)$ denote the range tripotent of $h$ in $E^{**}$. The following assertions are equivalent:
 \begin{enumerate}[$a)$]
     \item $T$ is orthogonality preserving;
     \item There exists a unital Jordan $^*$-homomorphism $S:M(\mathcal{A})\to E_2^{**}(r)$ such that $S(x)$ and $h$ operator commute in the JB$^*$-algebra $E_2^{**}(r)$ and \begin{equation}\label{eq fund equation conts OP JBSTAR}
         T(x)=h\circ_{r} S(x)= \{h,r,S(x)\}= U_{h^{\frac12}} (S(x)),
     \end{equation} for every $x\in M(\mathcal{A})$ where $h^{\frac12}$ is the square root of the positive element $h$ in the JB$^*$-algebra $E_2^{**} (r)$ and the $U$ operator is the one given by this JB$^*$-algebra {\rm(}in particular $T(\mathcal{A})\subseteq E_2^{**}(r)${\rm)};
     \item $T$ preserves zero-triple-products.
\end{enumerate}
\end{theorem}

%The statement concerning $U_{h^{\frac12}}$ in $(b)$ follows from the fact that $h$ is positive in $E_2^{**}(r)$  and operator commute with every element of the form $S(x)$.\smallskip

Let $\mathcal{A}$ be a JB$^*$-algebra whose positive part will be denoted by $\mathcal{A}^+$. Fix $a\in \mathcal{A}^+$. The JB$^*$-subalgebra of $\mathcal{A}$ generated by $a$, coincides with the JB$^*$-subtriple of $\mathcal{A}$ generated by $a,$ and will be denoted by the same symbol $\mathcal{A}_a$.\label{page ref subalgebra and subtriple generaed coincide} To see this we shall simply observe that the range tripotent of $a$ in $\mathcal{A}^{**}$ lies in the triple multipliers of the JB$^*$-subtriple generated by $a$, and thus $a^2=\{a,r(a),a\}\in \mathcal{A}_a$. It is well known that $\mathcal{A}_a$ is isometrically isomorphic to an abelian C$^*$-algebra (cf. \cite[3.2.3]{HOS}). Moreover, the JB$^*$-subalgebra of $\mathcal{A}$ generated by a self-adjoint element $b$ is isometrically JB$^*$-isomorphic to a commutative C$^*$-algebra (cf. \cite[The spectral theorem 3.2.4]{HOS}). We can define in this way a continuous functional calculus at the element $b$.

\begin{lemma}\label{l AkPed for JBstar} Let $I$ be a closed Jordan ideal of a JB$^*$-algebra $\mathcal{A}$. Suppose  $x+I , y +I $ are two positive orthogonal elements in $\mathcal{A}/I$. Then there exist $a,b\in I^{+}$ satisfying $(x-a)\perp (y-b)$.
\end{lemma}

\begin{proof} First observe that we can assume that $x,y$ are positive elements in $\mathcal{A}.$ Indeed, pick $a+I\in (\mathcal{A}/I)_{sa}$ (with $a\in \mathcal{A}_{sa}$) such that $x+I=(a+I)^2=a^2+I,$ and the same for $y$.\smallskip

Thought the rest of the proof literally follows the arguments in \cite[Proposition 2.3]{AkPed77}, an sketch of the ideas is included here for the sake of completeness. Define, via continuous functional calculus, $x_1=(x-y)_{+}$ and $y_1=(x-y)_{-}.$ By construction $x_1,y_1\geq 0$ and $x_1 y_1= 0$. Since $(x+I),(y+I)\geq 0$ with $(x+I) (y+I) =0$, we deduce that $((x+I)-(y+I))_{+}=x+I$, and hence $$ x_1+I=(x-y)_{+}+I=((x+I)-(y+I))_{+}=x+I.$$ Similarly we have $y_1+I=y+I.$ Setting $a=x-x_1$ and $b=y-y_1$ we get the desired conclusion.
\end{proof}

As it was pointed out in \cite[Section 2]{GarPeUnitCstaralg20}, Goldstein's charactersiation of orthogonal bilinear forms in \cite{Gold} turned out to be a very useful tool in the study of orthogonality preserving operators. The Jordan version of this result has been studied in \cite{JamPeSidd2015}. More concretely, let $V:\mathcal{A}\times \mathcal{A}\to \CC$ be a bilinear form on a JB$^*$-algebra $\mathcal{A}.$ Following \cite{JamPeSidd2015} we say that $V$ is orthogonal (respectively, orthogonal on $\mathcal{A}_{sa}$) if $V(a,b^*)=0$ for every $a,b$ in $\mathcal{A}$ (respectively, in $\mathcal{A}_{sa}$) with $a\perp b.$ Corollary 3.14 (Propositions 3.8 and 3.9) in \cite{JamPeSidd2015} shows that a bilinear form on a JB$^*$-algebra $\mathcal{A}$ is orthogonal if and only if $V$ is orthogonal on $\mathcal{A}_{sa}.$ Actually, it is not hard to see that the latter is equivalent to $V$ being orthogonal on $\mathcal{A}^+$. This result enables us to obtain an appropriate Jordan version of \cite[Proposition 1]{GarPeUnitCstaralg20}.

\begin{proposition}\label{p Ortpreserving on Asa JBstar}
Let $T:\mathcal{A}\to E$ be a bounded linear operator from a JB$^*$-algebra to a JB$^*$-triple. The following statements are equivalent:
\begin{enumerate}[$(a)$]
    \item $T$ preserves orthogonality;
    \item $T$ preserves orthogonality on $\mathcal{A}_{sa}$;
    \item $T$ preserves orthogonality on $\mathcal{A}^{+}.$
\end{enumerate}
\end{proposition}

\begin{proof} $(a) \Rightarrow (b) \Rightarrow (c)$ are clear. Now let us assume that $T:\mathcal{A}\to E$ preserves orthogonality on $\mathcal{A}^+.$ Let us fix $x\in E$ and $\phi \in E^*$ and define $V_{\phi}:\mathcal{A}\times \mathcal{A}\to \CC$ by $V_{\phi}(a,b):=\phi (\{T(a),T(b^*),x \}).$ By assumptions, $V_{\phi}$ is orthogonal on $\mathcal{A}^+.$ An analogous argument to that in \cite[Proposition 1]{GarPeUnitCstaralg20} shows that $V_{\phi}$ is orthogonal on $\mathcal{A}_{sa}$ (just apply that if $x\perp y$ in $\mathcal{A}_{sa}$, and we write these elements as differences of orthogonal positive elements $x = x^+-x^-$ and $y= y^+ - y^-$, we have $x^\sigma y^{\tau} =0$ for $\sigma, \tau\in \{\pm\}$). Thus, by \cite[Corollary 3.14, Propositions 3.8 and 3.9]{JamPeSidd2015} we deduce that $V_{\phi}$ is orthogonal on $\mathcal{A}$. Let us fix $a\perp b$ in $\mathcal{A}.$ Then $V_{\phi}(a,b^*)=\phi(\{T(a),T(b),x \} )=0. $ The Hahn-Banach theorem shows that $\{T(a),T(b),x\}=0.$ It follows from the arbitrariness of $x$ in $E$ that $L(T(a),T(b))=0,$ equivalently, $T(a)\perp T(b),$ witnessing that $T$ preserves orthogonality on $\mathcal{A}.$
\end{proof}

An element $a$ in a unital JB$^*$-algebra $\mathcal{A}$ is \emph{invertible} if there exists $b\in \mathcal{A}$ such that $a \circ b = 1$ and $a^2 \circ b = a.$ The element $b$ is unique and will be denoted by $a^{-1}$ (cf. \cite[3.2.9]{HOS} or \cite[Definition 4.1.2]{Cabrera-Rodriguez-vol1}). It is known that $a\in \mathcal{A}$ is invertible if and only if the mapping $U_a$ is invertible and in such a case $U_a^{-1} = U_{a^{-1}}$ \cite[Theorem 4.1.3]{Cabrera-Rodriguez-vol1}. For $a\in \mathcal{A}$ invertible, the mapping $M_a$ need not be, in general, invertible. However, if $a\in Z(\mathcal{A})$ is invertible, it can be easily checked that $M_{a^2} = U_a$ is invertible. If we further assume that $a$ is invertible and positive in  $Z(\mathcal{A})$, the mapping $M_a = U_{a^{\frac12}}$ is invertible too.\smallskip

Let $u$ be an element in $\mathcal{A}$. We say that $u$ is a \emph{unitary} if it is invertible in the Jordan sense and its inverse coincides with $u^*$ (cf. \cite[3.2.9]{HOS} and \cite[Definition 4.1.2]{Cabrera-Rodriguez-vol1}). Proposition 4.3 in \cite{BraKaUp78} assures that an element $u$ in $\mathcal{A}$ is a unitary if and only if it is a unitary in the JB$^*$-triple sense, that is, $\mathcal{A}_2 (u) = \mathcal{A}$. \smallskip

We have already gathered the tools required to establish a generalization of \cite[Lemma 1]{GarPeUnitCstaralg20}.

\begin{lemma}\label{l kernel is an ideal JBstar} Let $T : \mathcal{A}\to E$ be an orthogonality preserving bounded linear operator from a JB$^*$-algebra to a JB$^*$-triple. Let $h=T^{**} (1)\in E^{**}$, $r$ the range tripotent of $h$ in $E^{**}$, and let $S:
\mathcal{A} \to E^{**}$ denote the triple homomorphism given by Theorem \ref{thm characterization of OP JBstar}. Then $\ker(T) =\ker (S)$ is a norm closed (triple) ideal of $\mathcal{A}$. Moreover, the quotient mapping $\widehat{T}: \mathcal{A}/\ker(T)\to E,$ $\widehat{T}(x+\ker(T)) = T(x)$ is an  orthogonality preserving bounded linear mapping.
\end{lemma}

\begin{proof}
The relation $\ker(T) \supseteq \ker (S)$ is absolutely clear. Let us assume that $x\in \ker(T).$ Then $h\circ_r S(x)=0.$ We recall that $h$ is positive in $E_2^{**}(r(h)).$ Thus, by Lemma 4.1 in \cite{BurFerGarPe09}, $h\circ_r S(x)=0$ is equivalent to $h\perp S(x).$ Now, by \cite[Lemma 1]{BurFerGarMarPe2008} $S(x)\perp r(h)=r.$ Since $S(x)$ lies in $E_2^{**}(r)$ we have $S(x)=0.$ Therefore $\ker(S)=\ker(T).$ Having in mind that $S$ is a triple homomorphism we deduce that $\ker(S)$ is a norm closed triple ideal, and hence a Jordan ideal, of $\mathcal{A}$ (see comments in page \pageref{page Jordan and triple ideals coincide}).\smallskip

Clearly $\mathcal{A}/\ker(T)$ is a JB$^*$-algebra (cf. comments in \pageref{page Jordan and triple ideals coincide}) and $\widehat{T}$ is well defined. Proposition \ref{p Ortpreserving on Asa JBstar} shows that in order to see that $\widehat{T}$ is orthogonality preserving it suffices to show that $\widehat{T}$ preserves orthogonality on $\left(\mathcal{A}/\ker(T)\right)^{+}$. Let us take $x+\ker(T), y+\ker(T)\in \left(\mathcal{A}/\ker(T)\right)^{+}$ with $(x+\ker(T)) (y+\ker(T))=0$. Applying Lemma \ref{l AkPed for JBstar} we find $a,b\in \ker(T)^{+}$ satisfying $(x-a)\perp (y-b)$. By hypotheses $\widehat{T}(x+\ker(T)) =T(x-a)\perp T(y-b) = \widehat{T}(y+\ker(T))$.
\end{proof}

Let $a$ and $b$ be two hermitian elements in a JB$^*$-algebra $\mathcal{A}$. %Let us suppose that $a$ and $b$ operator commute.
We have already commented that the JB$^*$-subalgebra generated by $a$ and $b$ is a JC$^*$-subalgebra of some $B(H)$ \cite[Corollary 2.2]{Wright77}. We can further conclude that when $a$ and $b$ are regarded as elements in $B(H)$, they commute in the usual sense whenever they operator
commute in $\mathcal{A}$ \cite[Proposition 1]{Topping}. %Similarly, two elements $a$ and $b$ of $A_{sa}$ operator commute if and only if $ a^2\circ b=\{a,b,a\}$ (i.e., $a^2\circ b=2(a\circ b)\circ a-a^2\circ b$. If $b\in A$ we use $\{b\}'$ to denote the set of elements in $A$ that operator commute with $b$. (This corresponds to the usual notation in von Neumann algebras). Also, Lemma 5 in \cite{BurFerGarMarPe2008} states that if $a,b\in A_{sa}$ operator commute then $a\circ b$ and $a$ operator commute.

\begin{lemma}\label{l equation for multiplier JB}
Let $T:\mathcal{A}\to E$ be a bounded linear operator preserving orthogonality from a JB$^*$-algebra to a JB$^*$-triple. Let $h$, $r$ and $S$ be those given by Theorem \ref{thm characterization of OP JBstar}.
Then the identity \begin{equation}\label{eq identity multiplier JB} \{T(a), T(a),T(a)\}=h^{[3]} \circ_r S(a^{3}),
\end{equation} holds for every $a\in \mathcal{A}_{sa}$. Furthermore, the operator $\mathcal{A}\ni a \mapsto h^{[3]} \circ_r S(a)$ is $E$-valued.
\end{lemma}

\begin{proof}  We know that $h$ is positive in $ E_2^{**}(r)$, and hence $h^{[3]} = (h\circ_r h)\circ_r h = h^3$. Let us fix $a\in \mathcal{A}_{sa}.$ Proposition 4.1 in \cite{BurFerGarPe09} shows that $T(a)\in E_2^{**}(r)_{sa}$ and that the sets $\{T(a) , h\}$ and $\{ S(a), h\}$ are formed by pairs of elements which operator commute in $ E_2^{**}(r)$ (see also Theorem \ref{thm characterization of OP JBstar}). Therefore $S(a)$ and $h$ generate a JC$^*$-subalgebra of some $B(H),$ and they commute in the usual sense. Let us write $x\cdot y$ to denote the product of two elements $x,y$ in this particular $B(H)$ space. Then we have $x\circ_r y =\frac{1}{2}(x\cdot y+y\cdot x),$ and $T(a) = h\circ_r S(a) = h\cdot S(a).$ By the uniqueness of the triple product (cf. \cite[Proposition 5.5]{Ka}) we have
$$\{T(a),T(a),T(a)\}=T(a)\cdot T(a) \cdot T(a)=h\cdot S(a)\cdot h\cdot S(a)\cdot h\cdot S(a) $$ $$=h^3\cdot S(a)^3= h^3 \cdot S(a^3) =h^{[3]}\circ_r S(a^3), $$ which proves the desired identity.\smallskip

Let us fix $b\in \mathcal{A}_{sa}$ it follows from the previous identity that  $$h^{[3]}\circ_r S(b)=h^{[3]}\circ_r S((b^{[\frac{1}{3}]})^3)=\{T(b^{[\frac{1}{3}]}),T(b^{[\frac{1}{3}]}),T(b^{[\frac{1}{3}]})\}\in E,$$ thus $h^{[3]}\circ_r S(\mathcal{A}_{sa})\subseteq E,$ and hence the second statement follows.
\end{proof}

The next result is a consequence of the properties commented before the previous Lemma \ref{l equation for multiplier JB}.

\begin{lemma}\label{l lemma funcional calculus operator commute}
Lat $a$ and $b$ be two hermitian elements in a JB$^*$-algebra $\mathcal{A}.$ Suppose that $a$ and $b$ operator commute. Let $c$ and $d$ be elements in the JB$^*$-subalgebras of $\mathcal{A}$ generated by $a$ and $b$, respectively. Then the elements $c$ and $d$ operator commute.
\end{lemma}

%\begin{proof}
% Let $J$ be de JB$^*$-subalgebra of $A$ generated by $a$ and $b.$ Let us fix $n,m\in \NN$ and $x\in A.$ Since $a^n$ and $b^m$ are both self-adjoint, then $a^n$ and $b^m$ operator commute if and only if $\{a^n,a^n,b^m\}=a^{2n}\circ b^m.$ Let use denote by $x\cdot y$ the (associative) product of two elements in $B(H),$ where $B(H)$ is such that $J$ is a JC$^*$-subalgebra of $B(H).$ Having in mind that $a^n$ and $b^m$ commute in the usual sense in $B(H)$ we have \begin{equation} \label{eq operator commute powers}
%     a^{2n}\circ b^m=\frac{1}{2}(a^{2n}b^m+a^{2n}b^m)=\frac{1}{2}(a^n a^n b^m+b^ma^n a^n)=\{a^n,a^n,b^m\}
% \end{equation} proving that $a^n$ and $b^m$ operator commute. Now let $p(t)=\sum_{i=1}^{n}\alpha_i t^i $ and $q(t)=\sum_{j=1}^{m}\beta_i t^i $ be two polynomials with $p(0)=q(0)=0$ and fix $x\in A.$ We have
% $$ (p(a)\circ x)\circ q(b)=\sum_{i,j}(a^i\circ x)\circ b^j)=\sum_{i,j}a^i\circ ( x\circ b^j)=p(a)\circ (x\circ q(b))$$ thus $p(a)$ and $q(b)$ operator commute. Finally, by norm density of the elements of the form $p(a)$ (respectively, $q(b)$) in $Jord(a)$ (respectively, $Jord(b)$) and norm continuity of the product it follows that every element in $Jord(b)$ operator commutes with every element in $Jord(b)$.
% \end{proof}

Our next goal is a Jordan version of \cite[Propositions 2 and 3]{GarPeUnitCstaralg20}. The proof will require some translations to the Jordan terminology.

\begin{proposition}\label{p surjective OP preserves multipliers JB} Let $T : \mathcal{A}\to E$ be a surjective bounded linear operator preserving orthogonality from a JB$^*$-algebra to a JB$^*$-triple. Let $h=T^{**} (1)\in E^{**},$ $r$ the range tripotent of $h$ in $E^{**}$, and let $S: \mathcal{A}\to E_2^{**}(r)$ denote the Jordan $^*$-homomorphism given by Theorem \ref{thm characterization of OP JBstar}. Then the following statements hold:
\begin{enumerate}[$(a)$]
\item $r$ is a unitary in $E^{**};$
\item $h$ belongs to $M(E)$;
\item $h$ is invertible (and positive) in the JBW$^*$-algebra ${E}^{**}= E^{**}_2 (r)$;
\item $r$ belongs to $M(E)$, and consequently $E$ is a JB$^*$-algebra;
\item The triple homomorphism $S$ is $E$-valued and surjective;
\item If $x\in M(\mathcal{A})$ then $T^{**}(x)\in M(E) $;
\item The quotient mapping $\widehat{T}: \mathcal{A}/\ker(T)\to E,$ $\widehat{T}(x+\ker(T)) = T(x)$ is an orthogonality preserving bounded linear bijection;
\item There exist a triple homomorphism $S: \mathcal{A}\to E$ and a triple monomorphism $\widehat{S}: \mathcal{A}/\ker(S)\to E$ satisfying \begin{enumerate}[$(1)$] \item $\ker(T)= \ker(S)$;
    %\item $\widehat{S}^{**} (1+\overline{\ker(S)}^{w^*}) =S^{**} (1) = r;$
    \item $\widehat{S} (x+{\ker(S)}) =S(x)$;
    \item $\widehat{S}(x)$ and $h$ operator commute in $E^{**}_2(r)$, for all $x\in \mathcal{A};$
    \item $\widehat{T}(x+\ker(T)) =T(x) = h \circ_r \widehat{S}(x),$ %$=U_{h^{\frac12}} S(x),$
    for all $x\in \mathcal{A}$.%, where $h^{\frac12}$ denotes the square root of the positive element $h$ in the JB$^*$-algebra $E_2^{**} (r)$.
\end{enumerate}
\end{enumerate}
\end{proposition}

\begin{proof} $(a)$ Proposition 4.1 in \cite{BurFerGarPe09} proves that $T(\mathcal{A})\subseteq E^{**}_2 (r)$, and thus the surjectivity of $T$ implies that $E \subseteq E^{**}_2 (r) = P_2(r) (E^{**})$. Having in mind that $E$ is weak$^*$ dense in $E^{**}$, $E^{**}_2 (r)$ is weak$^*$-closed, and the triple product of $E^{**}$ is separately weak$^*$-continuous, we conclude that $E^{**} = E^{**}_2 (r),$ witnessing that $r$ is a unitary in $E^{**}$ and the latter is a JBW$^*$-algebra. %We can therefore regard $\mathcal{B}$ as a JB$^*$-subtriple of $\mathcal{B}_2^{**}(r)$ and consider the triple multiplier of $\mathcal{B}$ in $\mathcal{B}_2^{**}(r)$  $$M(\mathcal{B},\mathcal{B}_2^{**}(r))=\{ x\in \mathcal{B}_2^{**}(r) : \{x, \mathcal{B}, \mathcal{B}\}\subseteq  \mathcal{B}\}.$$
\smallskip

$(b)$ We shall next show that $h$ lies in $ M(E)$. Namely, since $T$ is surjective, it is enough to prove that $\{h,T(a),T(b)\}$ lies in $E,$ for every $a,b\in \mathcal{A}.$ Having in mind that $S$ and $T$ are symmetric operators from $\mathcal{A}$ into $E_2^{**}(r),$ Lemma \ref{l lemma funcional calculus operator commute} implies that for each $a\in \mathcal{A}_{sa}$ both $S(a)$ and $T(a)$ operator commute in $E_2^{**}(r)$ with every power of $h.$ Now for $a\in \mathcal{A}$ we write $a=a_1+ia_2$ with $a_1,a_2\in \mathcal{A}_{sa}$
$$(h^n\circ_r x)\circ_r S(a)=(h^n\circ_r x)\circ_r (S(a_1)+iS(a_2))=(h^n\circ_r x)\circ S(a_1)+i (h^n\circ_r x)\circ_r S(a_2) $$ $$=h^n\circ_r(x \circ_r S(a_1))+ih^n\circ_r (x\circ_rS(a_2))=h^n\circ_r (x\circ_rS(a)),$$ for all $x\in E^{**} = E^{**}_2 (r)$, witnessing that $h^n$ and $S(a)$ operator commute for all $n\in \NN$ and $a\in \mathcal{A}.$ Powers of $h$ are taken in the JB$^*$-algebra $E_2^{**}(r)$.  Clearly, the same holds for $T(a).$\smallskip

Now we claim that \begin{equation}\label{eq for h multiplier JBstar}
     \{h,T(a),T(b)\}=h^3\circ_r S(a^*\circ b)
\end{equation} holds for every $a,b\in \mathcal{A}.$ Indeed, for $a,b\in \mathcal{A}$ we have
 $$\{h,T(a),T(b)\}=(h\circ_r T(a)^{*_r})\circ_r T(a)+(T(b)\circ_r T(a)^{*_r})\circ_r h-(h\circ_r T(b))\circ_r T(a)^{*_r} $$
 $$=(h\circ_r T(a^*))\circ_r T(a)+(T(b)\circ_r T(a^*))\circ_r h-(h\circ_r T(b))\circ_r T(a^*) .$$
We shall show that each one of the summands on the right hand side equals $h^{[3]}\circ_r S(a^*\circ b).$
 $$(h\circ_r T(a^*))\circ_r T(b) =(h \circ_r (h \circ_r S(a^*) ))\circ_r T(b)=(h^2\circ_r S(a^*))\circ_r T(b)$$ $$ =S(a^*)\circ_r (h^2\circ_r T(b))= S(a^*)\circ_r (h^2\circ_r (h\circ S(b) ))=S(a^*)\circ_r (h^3\circ_r S(b))$$ $$=(S(a^*)\circ_r S(b))\circ_r h^3=h^3\circ_r S(a^*\circ b) = h^{[3]} \circ_r S(a^* \circ b). $$

Similar computations yield $$(T(b)\circ_r T(a^*))\circ_r h =h^{[3]}\circ _r S(a^*\circ b), \hbox{ and } (h\circ_r T(b))\circ_r T(a^*)=h^{[3]}\circ _r S(a^*\circ b),$$ which prove the claim.\smallskip

Since by Lemma \ref{l equation for multiplier JB} the operator $a\mapsto h^{[3]}\circ_r S(a)$ is $E$-valued, the equation \eqref{eq for h multiplier JBstar} shows that $h\in M(E).$\smallskip

$(c)$ By Lemma \ref{l kernel is an ideal JBstar} $\ker(T) = \ker(S),$ the quotient mapping $\widehat{T} : \mathcal{A}/\ker(T)\to E$, $\widehat{T} (x+\ker(T)) = T(x)$ is a bijective bounded linear operator preserving orthogonality from $\mathcal{A}/\ker(T)$ to $E$. The quotient mapping $\widehat{S} : \mathcal{A}/\ker(T)\to (E^{**},\circ_r,*_{r})$ is a Jordan $^*$-isomorphism, and $\widehat{T} (a+\ker(T)) = T(a) = h\circ_r S(a) = h\circ_r \widehat{S} (a+\ker(T))$ for all $a\in \mathcal{A}$.\smallskip

We shall employ a similar argument to that in the proof of \cite[Proposition 2$(c)$]{GarPeUnitCstaralg20}. By the arguments above we can assume that $T$ (and hence $T^{**}$) is a bijection. Let us find $c$ in $\mathcal{A}^{**}$ such that $T^{**} (c) =1$. We can also find, via Goldstine's theorem, a bounded net $(c_{\lambda})_{\lambda}$ in $\mathcal{A}$ converging to $c$ in the weak$^*$ topology of $\mathcal{A}$. Consequently, $(T(c_{\lambda}))_{\lambda}\to T^{**} (1)$ in the weak$^*$ topology of $E^{**}$. On the other hand, the net $(S(c_{\lambda}))_{\lambda}$ is bounded in $E^{**}$, by the Banach--Alaoglu's theorem, there exists a subnet $(S(c_{\mu}))_{\mu}$ converging to some $z\in E^{**}$ in the weak$^*$  topology of $E^{**}$. Obviously the subnet $(T(c_{\mu}))_{\mu}\to T^{**} (1)$ in the weak$^*$ topology of $E^{**}$. Since the identity $T(c_{\mu}) = h\circ_r S(c_{\mu})$ holds for every $\mu$, we deduce from the separate weak$^*$ continuity of the product of $E_2(r)^{**}$ and the above facts that $r = T^{**} (c) = h\circ_r z$.\smallskip

Having in mind that, for each $a\in \mathcal{A}$, $h$ and $S(a)$ operator commute in the JBW$^*$-algebra $E_2^{**}(r)$ and that $z = \w^*\hbox{-}\lim_{\mu} S(c_{\mu})$, we conclude that $h$ and $z$ operator commute in $E_2^{**}(r)$. By combining this fact with the identity $r = T^{**} (c) = h\circ_r z$, it follows that $h$ is invertible in $(E^{**},\circ_r,*_{r})$ with $h^{-1} = z\in E^{**}$.\smallskip

$(d)$ Let us observe that $h$ is positive and invertible in the JBW$^*$-algebra $E^{**}=E_2^{**}(r)$ (see $(a)$, $(b)$ and $(c)$). The inverse of $h$ and the unit element $r$ both lie in the JB$^*$-subalgebra of $E_2^{**}(r)$ generated by $h$. We recall that the JB$^*$-subalgebra of $E_2^{**}(r)$ generated by $h$ coincides with the JB$^*$-subtriple of $E_2^{**}(r)$ generated by $h$ (see the comments in page \pageref{page ref subalgebra and subtriple generaed coincide}). By applying that $M(E)$ is a JB$^*$-subtriple of $E^{**}=E_2^{**}(r)$ containing $h$, we deduce that $h^{-1}$ and $r$ both belong to $M(E)$. \smallskip

Finally it trivially follows from $r\in M(E)$ that for each $a,b\in E$ we have $E\ni \{a,r,b\} = a\circ_r b$ and $E\ni \{r,a,r\} =a^{*_r}$ (for the latter we observe that $E$ is a triple ideal of $M(E)$), and thus $E$ is a JB$^*$-subalgebra of $E^{**} = E_2^{**} (r)$.\smallskip

$(e)$ Having in mind that $h\in M(E)$ is positive and invertible in the JBW$^*$-algebra $E^{**}_2(r)= E^{**}$ and operator commutes with every element in the images of $S$ and $T$, and the identity in \eqref{eq fund equation conts OP JBSTAR}, the desired conclusion follows from the identity $$ S(a)=U_{h^{\frac12}}^{-1} T(a)= U_{h^{-\frac12}} T(a) = \{h^{-\frac{1}{2}},T(a)^*,h^{-\frac{1}{2}}\}\in E,$$ for all $a\in \mathcal{A}$. It can be also deduced from the fact that $$S(a) = h^{-1}\circ_r T(a) = \{h^{-1}, r, T(a)\}\in E,$$ for all $a\in \mathcal{A}$, because $r,h\in M(E)$ with $h$ positive and invertible, $h^{-1}\in M(E)$, $E$ is a triple ideal of $M(E)$ and $h$ (and hence h$^{-1}$) operator commutes with every element in the images of $S$ and $T$.\smallskip

$(f)$ Let us fix $x\in M(\mathcal{A})$ and $a\in \mathcal{A}.$ Let us recall that, by Theorem \ref{thm characterization of OP JBstar}, the identity $$T(x) = h\circ_r S(x) = S(x)\circ_r h,$$ holds for all $x\in \mathcal{A}$. We know from $(e)$ that $S: \mathcal{A}\to (E, \circ_r, *_r)$ is a Jordan $^*$-homomorphism.  By the separate weak$^*$-continuity of the Jordan product of ${E}^{**} = E_2^{**}(r)$, the weak$^*$ continuity of $T^{**},S^{**} :\mathcal{A}\to E^{**}$ and the weak$^*$ density of $\mathcal{A}$ in $\mathcal{A}^{**}$ we can easily deduce that $$T^{**}(x) = h\circ_r S^{**}(x) = S^{**}(x)\circ_r h, \hbox{ for all } x\in \mathcal{A}.$$ Similar arguments are valid to show that $S^{**}:\mathcal{A}^{**}\to E^{**} = E_2^{**}(r)$ is a Jordan $^*$-homomorphism.\smallskip

We observe that $h$ and every element in the image of $S^{**}$ operator commute in the JBW$^*$-algebra $E^{**}_2(r)= E^{**}$ (essentially because $h$ and every element in the image of $S$ operator commute, cf. Theorem \ref{thm characterization of OP JBstar}). Since for each $a\in \mathcal{A}_{sa}^{**}$, the element $S^{**} (a)$ belongs to the self-adjoint part of $E^{**}_2(r)= E^{**}$ and operator commutes with $h$, we deduce from \cite[Lemma 5]{BurFerGarMarPe2008} that $h \circ_r S^{**} (a)$ and $h$ operator commute for every $a\in \mathcal{A}_{sa}^{**}$, and by linearity, for all $a\in \mathcal{A}^{**}$.\smallskip

Now, since $h$ lies in $ M(E)$ and $x\in M(\mathcal{A})$, by applying the above facts, we have
$$\begin{aligned} T^{**}(x)\circ_r T(a) &= (h\circ_r S^{**}(x))\circ_r (h\circ_r  S(a)) = h\circ_r ( (h\circ_r S^{**}(x)) \circ_r  S(a)) \\
&= h\circ_r ( ( S(a) \circ_r S^{**}(x)) \circ_r h) = h\circ_r ( S^{**} (a\circ x) \circ_r h)\\
&=h\circ_r ( S (a\circ x) \circ_r h) = h\circ_r T(x\circ a)\in {E},
\end{aligned}$$ for all $a\in \mathcal{A}$.  The surjectivity of $T$ proves that  $T^{**}(x)\circ_r E\subseteq E.$ We have shown that $T^{**} (x)$ is a multiplier of the JB$^*$-algebra $M(E,\circ_r, *_{r})$. The uniqueness of the triple product (cf. \cite[Proposition 5.5]{Ka}) assures that $T^{**} (x)\in M(E).$\smallskip

Finally, $(g)$ follows from Lemma \ref{l kernel is an ideal JBstar} and $(h)$ is a straight consequence of the previous statements.
\end{proof}

We can now finish this section with a complete description of all bijective bounded linear operators preserving orthogonality between JB$^*$-algebras, which follows as a consequence of Propositions \ref{p surjective OP preserves multipliers JB} and \ref{p Ortpreserving on Asa JBstar}.

\begin{corollary}\label{c Characterization bd OP plus bijective Jordan} Let $T : \mathcal{A}\to E$ be a bijective bounded linear operator from a JB$^*$-algebra to a JB$^*$-triple. Let $h=T^{**} (1)\in E^{**}$ and let $r$ be the range tripotent of $h$ in $E^{**}$. Then the following statements are equivalent:\begin{enumerate}[$(a)$] \item $T$ is orthogonality preserving;
\item The elements $h$ and $r$ belong to $M(E)$ with $h$ invertible and $r$ unitary and there exists a Jordan $^*$-isomorphism $S: \mathcal{A}\to (E,\circ_{r},*_{r})$ such that $h$ lies in $Z(E^{**},\circ_{r},*_{r})$, and $T(x)=h\circ_{r(h)} S(x) = U_{h^{\frac12}} S(x),$ for every $x\in \mathcal{A}$, where $h^{\frac12}$ denotes the square root of the positive element $h$ in the JB$^*$-algebra $E_2^{**} (r)$;
\item $T$ is biorthogonality preserving;
\item $T$ is orthogonality preserving on $\mathcal{A}_{sa}$;
\item $T$ is orthogonality preserving on $\mathcal{A}^{+}$;
\item $T$ preserves zero-triple-products, i.e. $$\{a,b,c\}=0 \Rightarrow \{T(a),T(b),T(c)\}=0;$$
\item $T$ preserves zero-triple-products in both directions, i.e. $$\{a,b,c\}=0 \Leftrightarrow \{T(a),T(b),T(c)\}=0.$$
\end{enumerate}
\end{corollary}

\section{The centroid of a JB$^*$-triple in the study of one-parameter semigroups}\label{sec: one-parametri groups OP Jordan}

This section is aimed to establish a Jordan version of the description of one-parameter semigroups of orthogonality preserving operators on C$^*$-algebras developed in \cite[Theorem 3]{GarPeUnitCstaralg20}. To facilite the arguments we shall employ some results and terminology developed by S. Dineen and R. Timoney in the wider setting of JB$^*$-triples (see \cite{DiTi88}). According to the nomenclature in the just quoted paper, given a JB$^*$-triple $E$, the \emph{centroid}, $C(E),$ of $E$ is the set of all bounded linear operators $T:E\to E$ satisfying $$ T\{x,y,z\} = \{T(x), y,z\}, \hbox{ for all } x,y,z\in E.$$ It is known that $T\in B(E)$ lies in $C(E)$ if and only if $T$ commutes with all operators of the form $L(x,x)$ ($x\in E$) if and only if $T$ commutes with all operators of the form $L(x,y)$ ($x,y\in E$). It is further known that the centroid of $E$ coincides with the centralizer of the underlying Banach space $E$ in the sense of \cite{Beh79,Cunn67}, and it is precisely the center of the subalgebra of $B(E)$ generated by the Hermitian operators (cf. \cite[Theorem 2.8 and Corollary 2.10]{DiTi88}).\smallskip

We have already employed the center of a JB$^*$-algebra $\mathcal{A}$ in the previous section. The \emph{centroid} of $\mathcal{A}$ as JB$^*$-algebra is the set of all bounded linear operators $T:\mathcal{A} \to \mathcal{A}$ satisfying $T(x\circ y ) = T(x) \circ y$ for all $x,y\in \mathcal{A}$. To reassure the reader, we note that all notions and terminology are perfectly compatible. Actually, the centroid of $\mathcal{A}$ as JB$^*$-algebra coincides with the centroid of $\mathcal{A}$ as JB$^*$-triple \cite[Proposition 3.4]{DiTi88}, and moreover, a bounded linear operator $T$ on $\mathcal{A}$ belongs to the centroid if and only if $T= M_a$ for some element $a$ in the center of $M(\mathcal{A})$ \cite[Proposition 3.5]{DiTi88}. Actually, \cite[Proposition 3.5]{DiTi88} is only valid for unital JB$^*$-algebras, however, if $T$ is an element in the centroid of a non-unital JB$^*$-algebra $\mathcal{A}$, $T^{**}$ is an element in the centroid of $\mathcal{A}^{**}$, so there exists an element $a$ in the center of $\mathcal{A}^{**}$ such that $T^{**}(x) = M_{a} (x) = a\circ x,$ for all $x\in \mathcal{A}^{**}$. Since $T= T^{**}|_{\mathcal{A}}$ is $\mathcal{A}$-valued, the element $a$ belongs to $M(\mathcal{A})$.\smallskip

We shall need the following result gathered from different papers.

\begin{lemma}\label{l technical centroid with two different products} Let $u$ be a unitary in a unital JB$^*$-algebra $\mathcal{A}$.  Let $(\mathcal{A},\circ_u,*_{u})$ be the JB$^*$-algebra associated with $u$. The following statements hold:\begin{enumerate}[$(a)$] \item The identity $$\begin{aligned}\J xyz &= (x\circ y^*) \circ z + (z\circ y^*)\circ x - (x\circ z)\circ y^*  \\
&=  (x\circ_u y^{*_u}) \circ_u z + (z\circ_u y^{*_u})\circ_u x - (x\circ_u z)\circ_u y^{*_u} %\\
%&=  (x\circ_v y^{*_v}) \circ_v z + (z\circ_v y^{*_v})\circ_v x - (x\circ_v z)\circ_v y^{*_v},
\end{aligned}$$ holds for all $x,y,z\in \mathcal{A}$;
\item If $h\in Z(\mathcal{A},\circ_u,*_{u})$, and $M_h^u$ denotes the Jordan multiplication operator by $a$ in $(\mathcal{A},\circ_u,*_{u})$ {\rm(}i.e., $M_h^u (x) =a \circ_{u} x${\rm)},
the equality $\{ M_h^u (x) , y, z\} = M_h^u \{x,y,z\}$, holds for all $x,y,z\in \mathcal{A}.$
\end{enumerate}
\end{lemma}

\begin{proof} $(a)$ As we commented in \eqref{eq triple product JBstar algebra}, $\mathcal{A}$ and $(\mathcal{A},\circ_u,*_{u})$ are JB$^*$-algebras for the original norm and the corresponding Jordan products and involutions. Therefore, the identity mapping is a surjective linear isometry between these two JB$^*$-triples, and thus it follows from \cite[Proposition 5.5]{Ka} and \eqref{eq triple product JBstar algebra} that the desired identity holds.\smallskip

$(b)$ Since $h\in Z(\mathcal{A},\circ_u,*_{u})$, the mapping $M_h^u$ is a centralizer of the JB$^*$-triple $(\mathcal{A},\{.,.,.\})$ (cf. \cite[Proposition 3.5]{DiTi88} and $(a)$). Consequently, $$M_h^u \{x,y,z\} = \{ M_h^u (x) , y, z\} \hbox{ for every } x,y,z\in \mathcal{A}.$$
\end{proof}

We recall that a \emph{triple derivation} on a JB$^*$-triple $E$ is a linear mapping $\delta: E\to E$ satisfying the so-called \emph{triple Leibniz' rule}: $$\delta\{a,b,c\} = \{\delta(a) , b , c\} + \{a, \delta( b ) , c\} +  \{a,b,\delta(c)\} \ \ \ (a,b,c\in E).$$ Let us give an example. Fix two elements $a,b\in E$. By the Jordan identity, the mapping $\delta (a,b) := L(a,b)-L(b,a)$ is a triple derivation on $E$ and obviously continuous. It is known that every triple derivation on a JB$^*$-triple is automatically continuous (see \cite[Corollary 2.2]{BarFri90}). Furthermore, by the separate weak$^*$ continuity of the triple product of every JBW$^*$-triple, we can conclude that for each triple derivation $\delta$ on a JB$^*$-triple $E$, the bitranspose $\delta^{**} : E^{**}\to E^{**}$ is a triple derivation too.\smallskip

Under the terminology of JB$^*$-triples and triple derivations, Proposition 4 in \cite{GarPeUnitCstaralg20} and some related results in \cite{PedS88} admit the next extension, which, as we shall see below, is even more natural and complete in this wider setting.

\begin{lemma}\label{l -parameter group of iso on a JB*-triple} Let $\{U_t : t\in\mathbb{R}_0^+\}$ be a uniformly continuous one-parameter semigroup of surjective isometries on a JB$^*$-triple $E$. Then there exists a triple derivation $\delta : E\to E$ satisfying $U_t = e^{t \delta}$ for all $t\in \mathbb{R}$. Furthermore, for each triple derivation $\delta$ on $E$ the assignment $t\mapsto e^{t \delta}$ is a one-parameter semigroup of surjective isometries {\rm(}triple automorphisms{\rm)} on $E$.
\end{lemma}

\begin{proof} Let us find $R\in B(E)$ such that $T_t = e^{t R}$ for all $t\in \mathbb{R}$. Since every surjective isometry on $E$ is a triple isomorphism \cite[Proposition 5.5]{Ka}, for each real $t$ the mapping $T_t$ satisfies $$e^{t R} \{a,b,c \} =T_t \{a,b,c \} = \{ T_t(a), T_t(b), T_t(c) \} = \{e^{t R}(a),e^{t R}(b), e^{t R}(c)\} \ \ (\forall t\in \mathbb{R}).$$ Taking derivatives at $t=0$ we get $$R \{a,b,c \} = \{R(a),b,c \} +\{a,R(b),c \} + \{a,b,R(c) \}, \hbox{ for all } a,b,c\in E,$$ which proves that $R= \delta$ is the desired triple derivation.\smallskip

For the last statement let us observe that every triple derivation $\delta$ on $E$ is a dissipative mapping (cf. \cite[Theorem 2.1]{BarFri90}). By \cite[Corollary 10.13]{BonsDun73} $\|e^{t \delta}\|\leq 1$ for all $t\in \mathbb{R}^+_0$. Since $-\delta$ also is a triple derivation on $E$, the just quoted result implies that $\|e^{t \delta}\|\leq 1$ for all $t\in \mathbb{R}$, and thus $e^{t \delta}$ is a surjective isometry (equivalently, a triple automorphism) on $E$ for all $t\in \mathbb{R}$.
\end{proof}

Let $\mathcal{B}$ be a JB$^*$-algebra. A linear mapping $D : \mathcal{B} \to \mathcal{B}$ is said to be a \emph{Jordan derivation} if $D(a \circ b) = D(a) \circ b + a \circ D(b)$, for every $a,b$ in $\mathcal{B}$. A Jordan $^*$-derivation on $\mathcal{B}$ is a Jordan derivation $D$ satisfying $D(a^*) = D(a)^*$ for all $a\in \mathcal{B}$. If $1$ is a unit in $\mathcal{B}$, we have $D(1) =0$. Every Jordan derivation on a JB$^*$-algebra is automatically continuous (see \cite[Corollary 2.3]{HejNik96}). %It is also known that every Jordan derivation on a C$^*$-algebra is an associative derivation \cite[Theorem 6.3]{John96}.
Actually the results in \cite[Lemmata 1 and 2]{HoMarPeRu} (see also \cite[Proposition 3.7]{HoPeRu}) show that a linear mapping $\delta$ on a unital JB$^*$-algebra $\mathcal{B}$ is a triple derivation if and only if $\delta(1)^* = -\delta(1)$ and $\delta-\delta(\frac12 \delta(1), 1)$ is a Jordan $^*$-derivation on $\mathcal{B}$ (where $\delta(\frac12 \delta(1), 1) (x) = \delta(1)\circ x$ for all $x\in \mathcal{B}$), that is, $\delta$ is a triple derivation if and only if $\delta$ is the sum of a Jordan $^*$-derivation and a Jordan multiplication operator by a skew symmetric element in $\mathcal{B}$.\smallskip

We can continue with a Jordan version of one of the statements in \cite[Remark 2]{GarPeUnitCstaralg20}.

\begin{lemma}\label{l Jordan *-derivations vanish on the center} Let $d:\mathcal{A}\to \mathcal{A}$ be a Jordan $^*$-derivation on a JB$^*$-algebra. Then $d$ vanishes on the center of $\mathcal{A},$ and consequently $d$ commutes with $M_z$ for every $z\in Z(\mathcal{A})$.
\end{lemma}

\begin{proof} We have already justified in the comments prior to this lemma that $d$ is a triple derivation on $\mathcal{A}$. Therefore $d|_{\mathcal{A}_{sa}} : \mathcal{A}_{sa}\to \mathcal{A}_{sa}$ is a Jordan derivation on the JB-algebra $\mathcal{A}_{sa}$. Since the center of $\mathcal{A}$ is $^*$-invariant it suffices to prove that $d$ vanishes on every element in $Z(\mathcal{A})_{sa}= Z(\mathcal{A}_{sa}) = Z(\mathcal{A})\cap \mathcal{A}_{sa}.$\smallskip

The Approximation Theorem \cite[Approximation Theorem 4.2]{Upmeier0} guarantees that every Jordan derivation on a JB$^*$-algebra can be approximated in the strong operator topology by inner derivations (i.e. by derivations which are finite sums of maps of the form $x\mapsto [M_a,M_b] (x) = (M_a M_b - M_b M_a)(x)$). Fix $z\in Z(\mathcal{A})_{sa}$ and an arbitrary $\varepsilon>0$. It follows from the above result that there exist $a_1,b_1,\ldots,a_m,b_m\in \mathcal{A}_{sa}$ satisfying $\displaystyle \left\| d(z) - \sum_{j=1}^m [M_{a_j},M_{b_j}] (z)\right\|<\varepsilon.$ Having in mind that $z$ is central we obtain $[M_{a_j},M_{b_j}] (z) = a_j\circ (b_j\circ z) -  b_j\circ (a_j\circ z)= z\circ (b_j\circ a_j) -  z\circ (a_j\circ b_j) =0,$ for all $j$. It then follows that $\left\| d(z) \right\|<\varepsilon,$ and the arbitrariness of $\varepsilon>0$ implies that $d(z) =0$.\smallskip

An alternative proof can be obtained as follows: The hermitian part of $\mathcal{A}$ is a real JB$^*$-triple in the sense employed in \cite{HoMarPeRu} and $d|_{\mathcal{A}_{sa}} : \mathcal{A}_{sa}\to \mathcal{A}_{sa}$ is a triple derivation. By the Jordan identity, a typical example of a triple derivation on $\mathcal{A}_{sa},$ regarded as a real JB$^*$-triple, is one given by $\delta(a,b) (x) = L(a,b) (x) -L(b,a)(x) =\{a,b,x\}- \{b,a,x\}$ ($x\in \mathcal{A}_{sa}$), where $a,b$ are fixed elements in $\mathcal{A}_{sa}$. Let us pick $z\in Z(\mathcal{A}_{sa})$. It is easy to check that $$ \begin{aligned} \delta(a,b) (z) &=\{a,b,z\}- \{b,a,z\}=(a\circ b) \circ z + (z\circ b)\circ a  - (a\circ z)\circ b \\
&- (b\circ a) \circ z - (z\circ a)\circ b + (b\circ z)\circ a\\
&= 2  (z\circ b)\circ a  - 2 (a\circ z)\circ b = 2 (a\circ b) \circ z - 2 (b\circ a) \circ z =0.
\end{aligned}$$

Those triple derivations on a real JB$^*$-triple which are expressed as finite sums of triple derivation of the form $\delta(a,b)$ are called inner derivations. Theorem 5 in \cite{HoMarPeRu} proves that every triple derivation on a real JB$^*$-triple can be approximated by inner derivations with respect to the strong operator topology. Therefore, given $z\in Z(\mathcal{A}_{sa})$ and $\varepsilon>0$ there exist $a_1,\ldots,a_m$, $b_1,\ldots,b_m$ in $\mathcal{A}_{sa}$ such that $\displaystyle \varepsilon > \left\| d (z) - \sum_{j=1}^m \delta(a_j,b_j) (z) \right\| = \left\| d (z) \right\|.$ The arbitrariness of $\varepsilon>0$ assures that $d(z)=0$, as desired.\smallskip

Finally for $z\in Z(\mathcal{A})$ we have $$d M_z (x) = d (z\circ x) = d(z) \circ x + z\circ d(x) = M_z d(x), \hbox{ for all } x\in \mathcal{A}.$$
\end{proof}

We recall next the definition and basic properties of the strong$^*$ topology for general JBW$^*$-triples.
Let us suppose that $\varphi$ is a norm one functional in the predual, $M_*$, of a JBW$^*$-triple $M.$ If $z$ is any norm one element in $W$ with $\varphi (z) =1$, Proposition 1.2 in \cite{barton1987grothendieck} proves that the mapping
$$(x,y)\mapsto \varphi\J xyz$$ is a positive sesquilinear
form on $M,$ which does not depend on the choice of $z$. We find in this way prehilbertian
seminorms on $M$ given by $\|x\|_{\varphi}^2:= \varphi\J xxz,$ ($x\in M$). The \emph{strong*-topology} of $M$ is the topology  generated by the family $\{
\|\cdot\|_{\varphi}:\varphi\in {M_*}, \|\varphi \| =1 \}$ (cf.
\cite{BarFri90}). As in the setting of von Neumann algebras, the triple product of every JBW$^*$-triple is jointly strong$^*$-continuous on bounded sets (see \cite[Theorem]{RodPa91} and \cite[\S 4 and Theorem 9]{PeRo2001}). It is known that a linear map between JBW$^*$-triples is strong$^*$ continuous if and only if it is weak$^*$ continuous (cf. \cite[Corollary 3]{RodPa91} and \cite[page 621]{PeRo2001}).\smallskip

Let us go back to the triple spectrum. For each non-zero element $a$ in a JB$^*$-triple $E$ we set $m_q (a) := \min \{\lambda : \lambda\in \Omega_a\}$, where $\Omega_a$ denotes the triple spectrum of $a$. We set $m_q (0) =0$. The mapping $m_q : E\to \mathbb{R}_0^{+}$, $a\mapsto m_q(a)$ has been considered in \cite{JamPeSiddTah2015} in the study of the $\lambda$-function in the case of JBW$^*$-triples. One of the consequences of Theorem 3.1 in the just mentioned reference implies that \begin{equation}\label{eq mq is 1Lipschitz} |m_q (a) -m_q (b)| \leq \|a-b\|, \hbox{ for all } a,b\in E,
\end{equation} (cf. \cite[Theorem 3.1 and $(3.2)$]{JamPeSiddTah2015}).\smallskip

Let $z$ be an element in a JB$^*$-triple $E$. Back to the local Gelfand theory, we consider the JB$^*$-subtriples $E_z$ and $E_{z^{[3]}}$ generated by $z$ and $z^{[3]}$, respectively. It is known that $E_z = E_{z^{[3]}}$ (cf. \cite[comments before Proposition 2.1]{BunChuZal2000}). The following property holds:
\begin{equation}\label{eq uniqueness of cubic root}\hbox{ $z^{[3]} = a$ for some $a$ in $E$ implies that $z=a^{[\frac13]}\in E_a$.}
\end{equation} Namely, clearly $a\in E_z$ and thus $E_a \subseteq E_z$. On the other hand $z^{[3]}\in E_a$ and hence $E_z = E_{z^{[3]}} \subseteq E_a \subseteq E_z.$ The local Gelfand theory gives the statement.\smallskip

It is now the moment to describe the uniformly continuous one-parameter semigroups of orthogonality preserving operators on a general JB$^*$-algebra.

\begin{theorem}\label{t Wolff one-parameter for OP JBstar} Let $\mathcal{A}$ be a JB$^*$-algebra. Suppose $\{T_t: t\in \mathbb{R}_0^{+}\}$ is a family of orthogonality preserving bounded linear bijections on $\mathcal{A}$ with $T_0=Id$. For each $t\geq 0$ let $h_t = T_t^{**} (1)$, let $r_t$ be the range tripotent of $h_t$ in $\mathcal{A}^{**}$ and let $S_t: \mathcal{A} \to (\mathcal{A},\circ_{r_t},*_{r_t})$ denote the Jordan $^*$-isomorphism associated with $T_t$ given by Corollary \ref{c Characterization bd OP plus bijective Jordan}. Then the following statements are equivalent:\begin{enumerate}[$(a)$]\item $\{T_t: t\in \mathbb{R}_0^{+}\}$ is a uniformly continuous one-parameter semigroup of orthogonality preserving operators on $\mathcal{A}$;
\item $\{S_t: t\in \mathbb{R}_0^{+}\}$ is a uniformly continuous one-parameter semigroup of surjective linear isometries {\rm(}i.e. triple isomorphisms{\rm)} on $\mathcal{A}$ {\rm(}and hence there exists a triple derivation $\delta$ on $\mathcal{A}$ such that $S_t = e^{t \delta}$ for all $t\in \mathbb{R}${\rm)}, the mapping $t\mapsto h_t $ is continuous at zero, and the identity \begin{equation}\label{eq new idenity in the statement of theorem 1 on one-parameter Jordan} h_{t+s} =  h_t \circ_{r_t} S_t^{**} (h_s)= \{ h_t , {r_t},  S_t^{**} (h_s) \},
    \end{equation} holds for all $s,t\in \mathbb{R}.$
\end{enumerate}
\end{theorem}

\begin{proof} Let us begin with a common property employed in both implications. We claim that the identity in \eqref{eq new idenity in the statement of theorem 1 on one-parameter Jordan} %(i.e. $h_{t+s} =  h_t \circ_{r_t} S_t^{**} (h_s)$ for all $s,t\in \mathbb{R}$)
implies that \begin{equation}\label{eq 30 for ranges} r_{t+s} =  r_t \circ_{r_t} S_t^{**} (r_s)= S_t^{**} (r_s), \hbox{ for all } s,t\in \mathbb{R}.
\end{equation} Indeed, the elements
$h_{t}, h_{t}^{[\frac{1}{3^n}]}\in Z(\mathcal{A}^{**},\circ_{r_{t}},*_{r_{t}})$ and $S_t : \mathcal{A}\to (\mathcal{A},\circ_{r_t}, *_{r_t})$ is a Jordan $^*$-isomorphism, this can be applied to deduce that the identity \begin{equation}\label{eq 3n cubic roots h last theorem} h_{t+s}^{[\frac{1}{3^n}]} = h_{t}^{[\frac{1}{3^n}]} \circ_{r_t} S_t^{**}(h_s^{[\frac{1}{3^n}]}),
\end{equation} holds for all $n\in\mathbb{N}$. Let us briefly convince the reader. The uniqueness of the triple product (cf. \cite[Proposition 5.5]{Ka}) proves that $$\{x,y,z\} = (x \circ_{r_t} y^{*_{r_t}}) \circ_{r_t} z + (z \circ_{r_t} y^{*_{r_t}}) \circ_{r_t} x - (x \circ_{r_t} z) \circ_{r_t} y^{*_{r_t}}, \hbox{ for all } x,y,z\in \mathcal{A},$$ which combined with the commuting properties of $h_{t}$ and $h_{t}^{[\frac{1}{3^n}]}$ in $(\mathcal{A},\circ_{r_t}, *_{r_t})$ assures that  $$ \left(h_{t}^{[\frac{1}{3^{n+1}}]} \circ_{r_t} S_t^{**}(h_s^{[\frac{1}{3^{n+1}}]})\right)^{[3]} = \left(h_{t}^{[\frac{1}{3^{n+1}}]}\right)^{[3]} \circ_{r_t} S_t^{**}\left(h_s^{[\frac{1}{3^{n+1}}]}\right)^{[3]}   $$
$$ = h_{t}^{[\frac{1}{3^{n}}]} \circ_{r_t} S_t^{**}\left( \left(h_s^{[\frac{1}{3^{n+1}}]}\right)^{[3]}\right) = h_{t}^{[\frac{1}{3^{n}}]} \circ_{r_t} S_t^{**}\left(h_s^{[\frac{1}{3^{n}}]} \right) = h_{t+s}^{[\frac{1}{3^n}]},$$ where in the last equality we applied the induction hypothesis. The discussion in \eqref{eq uniqueness of cubic root} proves that $h_{t}^{[\frac{1}{3^{n+1}}]} \circ_{r_t} S_t^{**}(h_s^{[\frac{1}{3^{n+1}}]}) =  h_{t+s}^{[\frac{1}{3^{n+1}}]},$ which concludes the induction argument leading to \eqref{eq 3n cubic roots h last theorem}.\smallskip

Now, since $(h_{t+s}^{[\frac{1}{3^n}]})_n\to r_{t+s}$, $(h_{t}^{[\frac{1}{3^n}]})_n\to r_{t}$ and $(h_{s}^{[\frac{1}{3^n}]})_n\to r_{s}$ in the strong$^*$ topology of $\mathcal{A}^{**},$ $S_t^{**}$ is strong$^*$ continuous and the triple product of every JBW$^*$-triple is jointly strong$^*$ continuous on bounded sets (cf. \cite[Theorem]{RodPa91} and \cite[\S 4 and Theorem 9]{PeRo2001}), by taking strong$^*$ limits in \eqref{eq 3n cubic roots h last theorem} we get $r_{t+s} = r_t \circ_{r_t} S_t^{**} (r_s) = S_t^{**}(r_s),$ for all $s,t\in \mathbb{R}$, which concludes the proof of \eqref{eq 30 for ranges}. \smallskip

If we apply that $S_t : \mathcal{A}\to (\mathcal{A},\circ_{r_t}, *_{r_t})$ is a Jordan $^*$-isomorphism we also derive \begin{equation}\label{eq rt+s last theorem} r_{t+s}^{*_{r_t}}= S_t^{**}(r_s)^{*_{r_{t}}} = S_t^{**} (r_s^*), \hbox{ for all } s,t\in \mathbb{R}.
\end{equation}

$(a)\Rightarrow (b)$ It follows from the assumptions that $$h_{s+t} = T^{**}_{s+t} (1) = T^{**}_{t} (T^{**}_{s}(1)) = T^{**}_{t} (h_s) = h_t\circ_{r_t} S_t^{**} (h_s) =U^{t}_{h_t^{\frac12}} S_t^{**} (h_s),$$ for all $s,t\in\mathbb{R}$, where $U^{t}$ stands for the $U$ operator in the JB$^*$-algebra $(\mathcal{A},\circ_{r_t}, *_{r_t})$. We can apply \eqref{eq 30 for ranges} and \eqref{eq rt+s last theorem} to deduce $r_{t+s} = S_t^{**} (r_s),$ and $r_{t+s}^{*_{r_t}}= S_t^{**}(r_s)^{*_{r_{t}}} = S_t^{**} (r_s^*),$ for all $s,t\in \mathbb{R}$.\smallskip

It follows from the above conclusions that \begin{equation}\label{eq St Jordan *-isom last theorem} S_t: (\mathcal{A}, \circ_{r_s}, *_{r_{s}}) \to (\mathcal{A}, \circ_{r_{t+s}}, *_{r_{t+s}})
 \end{equation} is a (unital and isometric) triple isomorphism, and hence a Jordan $^*$-isomorphism.\smallskip

We know from Corollary \ref{c Characterization bd OP plus bijective Jordan} that each $h_t$ is a positive invertible element in $Z(\mathcal{A}, \circ_{r_t}, *_{r_{t}})$. Therefore, the mapping $M_{h_t}^{t} (x) := h_t\circ_{r_t} x$ is invertible in $B(\mathcal{A})$. If we fix an arbitrary $a\in \mathcal{A}$, we deduce from the hypotheses that
\begin{equation}\label{eq one 08032020} M_{h_{t+s}}^{t+s} S_{t+s} (a) = T_{t+s} (a) = T_t T_s (a) = M_{h_t}^{t} S_{t} M_{h_s}^s S_{s} (a)
\end{equation}
$$ = M_{h_t}^{t} S_{t} ({h_s}\circ_{r_s} S_{s} (a)) = M_{h_t}^{t} (S_t^{**} (h_s)\circ_{r_{t+s}} S_t S_s(a)),$$ where in the last equality we applied \eqref{eq St Jordan *-isom last theorem}. We focus next on the left-hand-side term in the first row and we expand it to get
\begin{equation}\label{eq two 08032020} M_{h_{t+s}}^{t+s} S_{t+s} (a) = \{h_{t+s}, r_{t+s}, S_{t+s} (a) \} = \hbox{(by \eqref{eq new idenity in the statement of theorem 1 on one-parameter Jordan})}
\end{equation}
$$= \{h_t \circ_{r_t} S_t^{**} (h_s), r_{t+s}, S_{t+s} (a) \}  = \{M_{h_t}^{t} S_t^{**} (h_s), r_{t+s}, S_{t+s} (a) \} $$ $$= M_{h_t}^{t} \{S_t^{**} (h_s), r_{t+s}, S_{t+s} (a) \}=  M_{h_t}^{t} ( S_t^{**} (h_s) \circ_{r_{t+s}} S_{t+s} (a)),$$
where in the penultimate step we applied Lemma \ref{l technical centroid with two different products}$(b)$ and the fact that $h_t$ belongs to $Z(\mathcal{A}, \circ_{r_t}, *_{r_{t}})$.\smallskip

The mapping $M_{h_t}^{t}$ is invertible in $B(\mathcal{A})$, we deduce from \eqref{eq one 08032020} and \eqref{eq two 08032020} that
\begin{equation}\label{eq three 08032020} M_{S_t^{**} (h_s)}^{t+s} S_{t+s} (a)  =  S_t^{**} (h_s) \circ_{r_{t+s}} S_{t+s} (a)
\end{equation} $$= S_t^{**} (h_s)\circ_{r_{t+s}} S_t S_s(a) =M_{S_t^{**} (h_s)}^{t+s} S_t S_s(a).$$

Since, by \eqref{eq St Jordan *-isom last theorem}, $S_t: (\mathcal{A}, \circ_{r_s}, *_{r_{s}}) \to (\mathcal{A}, \circ_{r_{t+s}}, *_{r_{t+s}})$ is a (unital) Jordan $^*$-isomorphism, and $h_s$ is positive, central and invertible in $(\mathcal{A}, \circ_{r_s}, *_{r_{s}})$, the element $S_t^{**} (h_s)$ is positive, central and invertible in $(\mathcal{A}, \circ_{r_{t+s}}, *_{r_{t+s}})$, and thus the mapping $M_{S_t^{**} (h_s)}^{t+s}$ is invertible in $B(\mathcal{A})$. It follows from \eqref{eq three 08032020} that $S_{t+s} (a)  = S_t S_s(a).$\smallskip

We have therefore shown that $\{S_s: s\in \mathbb{R}_0^{+}\}$ is a one-parameter semigroup of surjective linear isometries (i.e. triple isomorphisms) on $\mathcal{A}$. It only remains to show that it is  uniformly continuous. The uniform continuity of the semigroup $\{T_s: s\in \mathbb{R}\}$ proves that the mapping $t\mapsto h_s = T_s^{**} (1) $ is continuous at zero. \smallskip

For each real $s$, the element $h_s$ is positive, central and invertible in the JB$^*$-algebra $(\mathcal{A}, \circ_{r_s}, *_{r_{s}})$, and thus $m_q (h_s)>0$ for all $s\in \mathbb{R}$. Since $h_0 =1$, we can deduce from \eqref{eq mq is 1Lipschitz} the existence of $\rho>0$ and $0<\theta_1\leq \theta_2$ in $\mathbb{R}$ such that $\Omega_{h_s} \subseteq [\theta_1, \theta_2],$ equivalently $\theta_1\leq m_q (h_s) \leq \theta_2,$ for all $|s|<\rho$. In particular, for each natural $n$, the mapping $s\mapsto h_s^{[2 n-1]}$ is continuous at zero. Consequently, por each odd polynomial with zero constant term $p(\lambda)$, the mapping $s\mapsto p_{t}(h_s)$ also is continuous at zero (where we employ the triple polynomial calculus). Fix a natural $m$. By the Stone-Weierstrass theorem the function $g_m: [\theta_1, \theta_2] \to \mathbb{R}$, $g_m(\lambda )= \lambda^{\frac{1}{3^m}}$ can be uniformly approximated by an odd polynomial with zero constant term. By combining the previous facts we prove that the mapping $s\mapsto (g_m)_t(h_s) = h_{s}^{[\frac{1}{3^m}]}$ is continuous at zero (for all $m\in \mathbb{N}$).  Since the sequence $(g_m)_m$ converges uniformly to the unit element $\textbf{1}$ in $C[\theta_1, \theta_2],$ and $\textbf{1}_t (h_s) = r(h_s)$ for all $|s|<\rho$, it can be easily checked that the mapping $s\mapsto r(h_{s}) = r_s$ is continuous at zero.\smallskip

We can therefore conclude that the mapping $s\mapsto L(h_s,r_s) = M_{h_s}^{s}$ must be continuous at zero, where $M_{h_s}^{s}$ is an invertible element in $B(\mathcal{A})$. Having in mind that $S_{s} = \left(M_{h_s}^{s}\right)^{-1} T_{s}$ ($s\in \mathbb{R}_0^{+}$), we deduce that  $\{S_s: s\in \mathbb{R}_0^{+}\}$ is uniformly continuous one-parameter semigroup of surjective linear isometries, which finishes the proof of the first implication. \smallskip

$(b)\Rightarrow(a)$  The identity in \eqref{eq new idenity in the statement of theorem 1 on one-parameter Jordan} holds by assumptions, it then follows from \eqref{eq 30 for ranges} that $r_{t+s} =  S_t^{**} (r_s)$ for all $s,t\in \mathbb{R}$. As before, this implies that $$S_t: (\mathcal{A}, \circ_{r_s}, *_{r_{s}}) \to (\mathcal{A}, \circ_{r_{t+s}}, *_{r_{t+s}})$$ is a Jordan $^*$-isomorphism. Fix an arbitrary $a\in \mathcal{A}$ to compute $$\begin{aligned} T_{t+s} (a) & = h_{t+s} \circ_{r_{t+s}} S_{t+s} (a) = \{ h_{t+s}, r_{t+s}, S_{t+s} (a)\} = \{ h_{t+s}, r_{t+s}, S_{t} S_{s} (a)\} \\
& = \{  h_t \circ_{r_t} S_t^{**} (h_s) , S_{t}^{**} (r_{s}), S_{t} S_{s} (a)\} = \{ M_{h_t}^{t} S_t^{**} (h_s) , S_{t}^{**} (r_{s}), S_{t} S_{s} (a)\} \\
&=  M_{h_t}^{t}\{ S_t^{**} (h_s) , S_{t}^{**} (r_{s}), S_{t} S_{s} (a)\} =  M_{h_t}^{t} S_t^{**} \{ h_s , r_{s}, S_{s} (a)\} \\
& =  {h_t}\circ_{r_t} S_t^{**} \left( h_s \circ_{r_{s}} S_{s} (a)\right) = T_t^{**}  T_s^{**} (a) = T_t  T_s (a),
\end{aligned},$$ where in the fourth and sixth equalities we applied \eqref{eq new idenity in the statement of theorem 1 on one-parameter Jordan} and Lemma \ref{l technical centroid with two different products}$(b)$ with $h_t\in Z(\mathcal{A}, \circ_{r_t}, *_{r_t})$, respectively. We have proved that $\{T_t: t\in \mathbb{R}_0^+\}$ is a one-parameter semigroup of orthogonality preserving operators on $\mathcal{A}$. The uniform continuity of the semigroup can be easily deduced from the corresponding property of the one-parameter semigroup $\{S_t: t\in \mathbb{R}_0^+\},$ the continuity of the mapping $t\mapsto h_t$ at zero, and the identity $T_t (\cdot) = \{h_t ,r_t , S_t (\cdot)\}$ with the same arguments we gave in the final part of the proof of $(a)\Rightarrow (b)$.
\end{proof}

As in the case of C$^*$-algebras, in the above Theorem \ref{t Wolff one-parameter for OP JBstar} the sets $\{r_t: t\in \mathbb{R}_0^+\}$ and $\{h_t: t\in \mathbb{R}_0^+\}$ need not be one-parameter semigroups for the Jordan product (cf. \cite[Remark 3]{GarPeUnitCstaralg20}). However, assuming that each $T_t$ is a symmetric mapping these sets become semigroups for the Jordan product.

\begin{corollary}\label{c Wolff one-parameter for OP JBstar symmetric} Let $\mathcal{A}$ be a JB$^*$-algebra. Suppose $\{T_t: t\in \mathbb{R}_0^{+}\}$ is a family of symmetric orthogonality preserving bounded linear bijections on $\mathcal{A}$ with $T_0=Id$. For each $t\geq 0$ let $h_t = T_t^{**} (1)$, let $r_t$ be the range tripotent of $h_t$ in $\mathcal{A}_{sa}^{**}$ and let $S_t: \mathcal{A} \to (\mathcal{A}^{**},\circ_{r_t},*_{r_t})$ denote the Jordan $^*$-isomorphism associated with $T_t$ given by Corollary \ref{c Characterization bd OP plus bijective Jordan}. The following statements are equivalent:\begin{enumerate}[$(a)$]\item $\{T_t: t\in \mathbb{R}_0^{+}\}$ is a uniformly continuous one-parameter semigroup of orthogonality preserving operators on $\mathcal{A}$;
\item $\{S_t: t\in \mathbb{R}_0^{+}\}$ is a uniformly continuous one-parameter semigroup of surjective linear isometries {\rm(}i.e. triple isomorphisms{\rm)} on $\mathcal{A}$ {\rm(}and hence there exists a triple derivation $\delta$ on $\mathcal{A}$ such that $S_t = e^{t \delta}$ for all $t\in \mathbb{R}${\rm)}, the mapping $t\mapsto h_t $ is continuous at zero, $h_t,r_t\in Z(M(\mathcal{A}))$ and the identities
    $$h_{t+s} = h_t \circ h_s, \ r_{t+s} = r_t \circ r_s, \hbox{ and }$$
    $$h_{t+s} =  h_t \circ_{r_t} S_t^{**} (h_s)= \{ h_t , {r_t},  S_t^{**} (h_s) \},$$
    hold for all $s,t\in \mathbb{R}.$
\end{enumerate}

\noindent Moreover, if any of the previous equivalent statements holds, then there exist $h\in Z(M(\mathcal{A}))$ and a Jordan $^*$-derivation $d$ on $\mathcal{A}$ such that $r_t\circ S_t = e^{t d},$ and $$T_t (a) = e^{t h}\circ (r_t\circ S_t) (a) = e^{t (M_h+ d)} (a),$$ for all $a\in A$, $t\in \mathbb{R}.$
\end{corollary}

\begin{proof} We shall only prove the extra affirmations in $(a)\Rightarrow (b)$. We begin by observing that the additional hypothesis on $T_{t}$ --i.e. $T_t$ symmetric-- shows that $h_t,r_t\in M(\mathcal{A}_{sa})$. Therefore $r_t$ is a symmetric unitary in $\mathcal{A}^{**}$, and hence %we can find two orthogonal projections $p_t,q_t\in \mathcal{A}^{**}$ such that $r_t = p_t - q_t$ and $1= p_t + q_t$.
$r_t^2 = 1$. In particular, $S_t: \mathcal{A}\to \mathcal{A}$ is a symmetric operator too. Namely, we have shown in the comments before Lemma \ref{l kernel is an ideal JBstar} that the mapping $M_{h_t}^t (x) = h_t \circ_{r_t} x$ is invertible in $B(\mathcal{A})$. Then, by the symmetry of $T_t$, $h_t$ and $r_t$, we get $$\begin{aligned}M_{h_t}^t S_t(x^*) &=  h_t \circ_t S_t(x^*)= T_t(x^*) = T_t (x)^* = \left( h_t \circ_t S_t(x) \right)^*\\
&= \{ h_t, r_t, S_t(x)\}^* =  \{ h_t, r_t, S_t(x)^*\} = M_{h_t}^t \left(S_t(x)^*\right),
\end{aligned} $$ and thus $S_t(x)^* = S_t(x)^*$ for all $x\in \mathcal{A}$.
\smallskip

As in the proof of Theorem \ref{t Wolff one-parameter for OP JBstar}$(a)\Rightarrow (b)$, the identity $h_{t+s} =  h_t \circ_{r_t} S_t^{**} (h_s)$ implies that $r_{t+s} =   S_t^{**} (r_s),$ for all $s,t\in \mathbb{R}$ (cf. \eqref{eq 30 for ranges}).\smallskip

Since $S_t : \mathcal{A}\to (\mathcal{A},\circ_{r_t}, *_{r_t})$ is a Jordan $^*$-isomorphism we get $$\begin{aligned}
r_t^2 \circ S_t (a)= \{r_t, r_t,  S_t (a) \} &= S_t (a)  = S_t (a)^{*_{r_t}} = \{r_t,  S_t (a), r_t \} %= 2 (r_t\circ  S_t (a)^*)\circ  r_t - r_t^2 \circ S_t (a)^* \\
%&= 2 (r_t\circ  S_t (a))\circ  r_t - r_t^2 \circ S_t (a),
\end{aligned}$$ for all $a\in \mathcal{A}_{sa}$, which guarantees that $r_t$ and $S_t(a)$ operator commute for all $a\in \mathcal{A}_{sa}$ (compare \eqref{eq ideintity for operator commutativity of two hermitian elements}). Having in ming that $S_t (\mathcal{A}_{sa}) = \mathcal{A}_{sa}$ we conclude that $r_t$ lies in the center of $M(\mathcal{A})$.\smallskip

Now, by \cite[Corollary 4.1$(a)$]{BurFerGarPe09} for each $a\in \mathcal{A}_{sa}$ we have $$\begin{aligned}h_t^2 \circ T_t(a)  &=\{T_t (a),h_t,h_t\} =\{ h_t, T_t (a), h_t\}, %= 2 (h_t\circ T_t(a)^*) \circ h_t - h_t^2 \circ T_t(a)^*\\
%& = 2 (h_t\circ T_t(a)) \circ h_t - h_t^2 \circ T_t(a),
\end{aligned} $$ which combined with the surjectivity of $T_t$ and \eqref{eq ideintity for operator commutativity of two hermitian elements}, it suffices to deduce that $h_t$ lies in the center of $M(\mathcal{A})$.\smallskip

We shall next show that $\{r_t \circ S_t\}_{t\in \mathbb{R}}$ is a one-parameter group of Jordan $^*$-isomorphisms on $\mathcal{A}$. Indeed, since $r_t\circ S_t: \mathcal{A}\to \mathcal{A}$ is a Jordan $^*$-isomorphism, $r_s\in Z(\mathcal{A}^{**})$ and $r_{t+s} =   S_t^{**} (r_s)$ we get
$$(r_t\circ S_t) (r_s\circ S_s (a)) = (r_t\circ S_t^{**}) (r_s) \circ (r_t\circ S_t)(S_s(a)) = r_t^2 \circ (S_t^{**} (r_s)\circ S_t S_s(a))$$ $$= S_t^{**} (r_s)\circ S_t S_s(a)= r_{t+s}\circ S_t S_s(a) = (r_{t+s}\circ S_{t+s}) (a), $$ for all  $a\in A$, which proves the desired statement.\smallskip

Therefore $\{(r_t \circ S_t)|_{Z(\mathcal{A})} \}_{t\in \mathbb{R}}$ is a one-parameter group of Jordan $^*$-isomor-phisms on the commutative C$^*$-algebra $Z(\mathcal{A})$, and thus Lemma \ref{l Jordan *-derivations vanish on the center} guarantees that it must the identity constant group.
Finally, since $r_t,h_t\in Z(M(\mathcal{A}))$ we deduce that $$\begin{aligned}
h_{t+s} &= T_{t+s}^{**} (1) = T_{t}^{**} T_{s}^{**} (1)= h_t \circ_{r_t} S_t^{**} (h_s) \\
&=\{ h_t, {r_t},  S_t^{**} (h_s)\} = h_t\circ ({r_t}\circ  S_t^{**} (h_s)) = h_t\circ h_s
\end{aligned}$$ and consequently $r_{t+s} =  r_t\circ r_s,$ for all $t,s\in \mathbb{R}.$\smallskip

We have proved that $\{ r_t ; t\in \mathbb{R}\}$ and $\{ h_t ; t\in \mathbb{R}\}$ are uniformly continuous semigroups in $Z(M(\mathcal{A}))$. Since $Z(M(\mathcal{A}))$ is a unital commutative C$^*$-algebra, we can proceed as in the proof of \cite[Corollary 2]{GarPeUnitCstaralg20} to find $h\in Z(M(\mathcal{A}))$ such that $h_t = e^{t h}$ for all $t\in \mathbb{R}$. Having in mind that $\{r_t \circ S_t\}_{t\in \mathbb{R}}$ is a uniformly continuous one-parameter group of Jordan $^*$-isomorphisms on $\mathcal{A}$, Lemma \ref{l -parameter group of iso on a JB*-triple} and subsequent comments assure the existence of a Jordan $^*$-derivation $d$ on $\mathcal{A}$ such that $r_t\circ S_t = e^{t d},$ and thus $$T_t (a) = e^{t h}\circ (r_t\circ S_t) (a) = e^{t h}\circ e^{t d} (a) = e^{t (M_h+ d)} (a),$$ for all $a\in A$, $t\in \mathbb{R}$, where in the last equality we applied that $h \in Z(M(\mathcal{A}))$ and Lemma \ref{l Jordan *-derivations vanish on the center}.
\end{proof}

\textbf{Acknowledgements} A.M. Peralta partially supported by the Spanish Ministry of Science, Innovation and Universities (MICINN) and European Regional Development Fund project no. PGC2018-093332-B-I00, Junta de Andaluc\'{\i}a grant FQM375 and Proyecto de I+D+i del Programa Operativo FEDER Andalucia 2014-2020, ref. A-FQM-242-UGR18. \smallskip

\section*{Conflict of interest}
The authors declare that they have no conflict of interest.


\begin{thebibliography}{}

\bibitem{AkPed77} C.A. Akemann, G.K. Pedersen, Ideal perturbations of elements in C$^*$-algebras, \emph{Math. Scand.} \textbf{41}, no. 1, 117--139 (1977).

%\bibitem{AlBreExVill09} J. Alaminos, M. Bresar, J. Extremera, A. Villena, Maps preserving zero products, \emph{Studia Math.} \textbf{193}, no. 2, 131--159 (2009).

%\bibitem{ArJar03} J. Araujo, K. Jarosz, Biseparating maps between operator algebras, \emph{J. Math. Anal. Appl.} \textbf{282}, 48--55 (2003).

%\bibitem{Arendt} W. Arendt, Spectral properties of Lamperti operators, \emph{Indiana Univ. Math. J.} \textbf{32}, no. 2, 199--215 (1983).

%\bibitem{Arens51} R. Arens, The adjoint of a bilinear operation, \emph{Proc. Amer. Math. Soc.} \textbf{2}, 839-848 (1951).
%

%\bibitem{Aupetit00} B. Aupetit, Spectrum-preserving linear mappings between Banach algebras or Jordan-Banach algebras, \emph{J. London Math. Soc.} (2) \textbf{62}, no. 3, 917--924 (2000).

%\bibitem{AyuKudPe2014} S. Ayupov, K. Kudaybergenov, A.M. Peralta, A survey on local and 2-local derivations on C$^*$- and von Neuman algebras, \emph{Contemporary Mathematics, Amer. Math. Soc.} Volume \textbf{672}, 73-126 (2016).

\bibitem{barton1987grothendieck} T. Barton, Y. Friedman, Grothendieck's inequality for {$JB^*$}-triples and applications, {\em J. London Math. Soc.} (2) \textbf{36}, 3, 513--523 (1987).

\bibitem{BarFri90} T.~{Barton}, Y.~{Friedman}, {Bounded derivations of $JB\sp*$-triples,} {\em {Q. J. Math. Oxf. II. Ser.}}, \textbf{41}(163), 255--268 (1990).

\bibitem{BaTi} T. Barton, R.M. Timoney, Weak{$^\ast$}-continuity of {J}ordan triple products and its applications, {\em Math. Scand.} \textbf{59}, 2, 177--191 (1986).

\bibitem{Beh79} E. Behrends, \emph{$M$-structure and the Banach-Stone theorem}, {Lecture Notes in Math.,} vol. \textbf{736}. Springer, Berlin 1979.

\bibitem{BonsDun73} F.F. Bonsall, J. Duncan, \emph{Complete Normed Algebras}, Erg. der Math. 80,. Springer-Verlag, Berlin 1973.


\bibitem{BratRob1987} O. Bratteli, D.W. Robinson, \emph{Operator algebras and quantum statistical mechanics. 1. C*- and W*-algebras, symmetry groups, decomposition of states}. Second edition. Texts and Monographs in Physics. Springer-Verlag, New York, 1987.

\bibitem{BraKaUp78} R. Braun, W. Kaup, H. Upmeier, A holomorphic characterisation of Jordan-C$^*$-algebras, \emph{Math. Z.} \textbf{161}, 277--290 (1978).

%\bibitem{Bres07} M. Bre\v{s}ar, Characterizing homomorphisms, multipliers and derivations in rings with idempotents, \emph{Proc. Roy. Soc. Edinb. Sect. A.} \textbf{137}, 9-21 (2007).

%\bibitem{Bres12} M. Bre\v{s}ar, Multiplication algebra and maps determined by zero products, \emph{Linear and Multilinear Algebra} \textbf{60}, 763-768 (2012).

%\bibitem{BuShuWong2016} Q. Bu, M.-H. Hsu, N.-C. Wong, Orthogonally additive holomorphic maps between C$^*$-algebras, \emph{Studia Math.} \textbf{234}, 195-216 (2016).

\bibitem{BuChu92} L.J. Bunce, Ch.-H. Chu,  Compact  operations, multipliers and Radon-Nikodym property in JB$^*$-triples, \emph{Pacific J. Math.} \textbf{153}, 249--265 (1992).

\bibitem{BunChuZal2000} L.J. Bunce, Ch.-H. Chu, B. Zalar, Structure spaces and decomposition in JB$^*$-triples, \emph{Math. Scand.} \textbf{86}, no. 1, 17--35 (2000).

%\bibitem{Bur13} M. Burgos, Orthogonality preserving linear maps on C$^*$-algebras with non-zero socles, \emph{J. Math. Anal. Appl.} \textbf{401}, no. 2, 479--487 (2013).

%\bibitem{BurCabSanPe2016} M. Burgos, J. Cabello-S{\'a}nchez, A.M. Peralta, Linear maps between C$^*$-algebras that are $^*$-homomorphisms at a fixed point, \emph{Quaestiones Mathematicae} \textbf{42}(2): 151-164 (2019).

%\bibitem{BurFerGarPe2014} M. Burgos, F.J. Fern�ndez-Polo, J. Garc�s, A.M. Peralta, Local triple derivations on C$^*$-algebras, \emph{Comm. Algebra} \textbf{42}, no. 3, 1276-1286  (2014).

\bibitem{BurFerGarMarPe2008} M. Burgos, F.J. Fern{\' a}ndez-Polo, J.J. Garc{\'e}s, J. Mart{\'i}nez Moreno, A.M. Peralta, Orthogonality preservers in C$^*$-algebras, JB$^*$-algebras and JB$^*$-triples, \emph{J. Math. Anal. Appl.} \textbf{348}, 220--233 (2008).

%\bibitem{BurFerPe2014} M. Burgos, F.J. Fern�ndez-Polo, A.M. Peralta, Local triple derivations on C$^*$-algebras and JB$^*$-triples, \emph{Bull. Lond. Math. Soc.} \textbf{46}, no. 4, 709-724 (2014).

\bibitem{BurFerGarPe09} M. Burgos, F.J. Fern{\'a}ndez-Polo, J.J. Garc{\'e}s, A.M. Peralta, Orthogonality preservers revisited, \emph{Asian-Eur. J. Math.} \textbf{2}, 387--405 (2009).

%\bibitem{BurGarPe11} M. Burgos, J. Garces, A.M. Peralta, Automatic continuity of biorthogonality preservers between compact C$^*$-algebras and von Neumann algebras, \emph{J. Math. Anal. Appl.} \textbf{376} 221--230 (2011).

%\bibitem{BurGarPe11StudiaTriples} M. Burgos, J. Garces, A.M. Peralta, Automatic continuity of biorthogonality preservers between weakly compact JB$^*$-triples and atomic JBW$^*$-triples, \emph{Studia Math.} \textbf{204}, no. 2, 97--121 (2011).

%\bibitem{BurKaMoPeRa}  M. Burgos, A. Kaidi, A. Morales, A. M. Peralta, M. Ram{\'\i}rez, von Neumann regularity and quadratic conorms in JB$^*$-triples and C$^*$-algebras, \emph{Acta Mathematica Sinica} \textbf{24}, No. 2, 185-200 (2008).

%\bibitem{BurSanchOr} M. Burgos, J. S{\'a}nchez-Ortega, On mappings preserving zero products, \emph{Linear Multilinear Algebra} \textbf{61}, no. 3, 323--335 (2013).

\bibitem{Cabrera-Rodriguez-vol1}
{\sc Cabrera~Garc\'{\i}a, M., Rodr\'{\i}guez~Palacios, A.}
\newblock {\em Non-associative normed algebras. {V}ol. 1}, vol.~154 of {\em
  Encyclopedia of Mathematics and its Applications}.
\newblock Cambridge University Press, Cambridge, 2014.
\newblock The Vidav-Palmer and Gelfand-Naimark theorems.

\bibitem{Cabrera-Rodriguez-vol2}
{\sc Cabrera~Garc\'{\i}a, M., Rodr\'{\i}guez~Palacios, A.}
\newblock {\em Non-associative normed algebras. {V}ol. 2}, vol.~167 of {\em
  Encyclopedia of Mathematics and its Applications}.
\newblock Cambridge University Press, Cambridge, 2018.
\newblock Representation theory and the Zel'manov approach.


%\bibitem{CheKeLeeWong03} M.A. Chebotar, W.-F.Ke, P.-H. Lee, N.-C.Wong, Mappings preserving zero products, \emph{Studia Math.} \textbf{155}, 77--94 (2003).


%\bibitem{CivYoo65} P. Civin, B. Yood, Lie and Jordan structures in Banach algebras, \emph{Pacific J. Math.} \textbf{15}, 775--797 (1965).

%\bibitem{Dales00} H.G. Dales, \emph{Banach algebras and automatic continuity}. London Mathematical Society Monographs (New Series), Volume 24, Oxford Science Publications. The Clarendon Press, Oxford University Press, New York, 2000.

\bibitem{Cunn67} F. Cunningham, $M$-structure in Banach spaces, \emph{Math. Proc. Cambridge Philos. Soc.} \textbf{63}, 613--629 (1967).

\bibitem{Di86} S. Dineen, The second dual of a JB$^*$-triple system, In: Complex analysis, functional analysis and approximation theory (ed. by J. M\'ugica), 67-69, (North-Holland Math. Stud. 125), North-Holland, Amsterdam-New York, (1986).

\bibitem{DiTi88} S. Dineen, R. Timoney, The centroid of a JB$^*$-triple system, \emph{Math. Scand.} \textbf{62}, no. 2, 327--342 (1988).

\bibitem{Ed80} C.M. Edwards, Multipliers of JB-algebras, \emph{Math. Ann.}, \textbf{249}, 265--272 (1980).

\bibitem{EdRu96} C.M. Edwards, G.T. R\"{u}ttimann, Compact tripotents in bi-dual ${\rm JB}\sp *$-triples, \emph{Math. Proc. Cambridge Philos. Soc.} \textbf{120}, no. 1, 155--173 (1996).

%\bibitem{EssPe2018} A.B.A. Essaleh, A.M. Peralta, Linear maps on C$^*$-algebras which are derivations or triple derivations at a point, \emph{Linear Algebra and its Applications} \textbf{538}, 1-21 (2018).

%\bibitem{FaaGhahra2019} B. Fadaee, H. Ghahramani, Linear maps behaving like derivations or anti-derivations at orthogonal elements on C$^*$-algebras, to appear in \emph{Bull. Malays. Math. Sci. Soc.} https://doi.org/10.1007/s40840-019-00841-6. arXiv:1907.03594v1

%\bibitem{Fack82} Th. Fack, Finite sums of commutators in C$^*$-algebras, \emph{Ann. Inst. Fourier, Grenoble} \textbf{32}, 129-137 (1982).

%\bibitem{FackdelaHarpe80} Th. Fack, P. de la Harpe, Sommes de commutateurs dans les alg'ebres de von Neumann finies continues, \emph{Ann. Inst. Fourier, Grenoble} \textbf{30}, 49-73 (1980).

\bibitem{FriRu85} Y. Friedman, B. Russo, Structure of the predual of a JBW$^*$-triple, \emph{J. Reine u. Angew. Math.} \textbf{356}, 67--89 (1985).

%\bibitem{friedman1985solution} Y. Friedman, B. Russo, Solution of the contractive projection problem, {\em J. Funct. Anal.} \textbf{60}, 1, 56--79 (1985).


%\bibitem{GhaPan2018} H. Ghahramani, Z. Pan, Linear maps on $^*$-algebras acting on orthogonal elements like derivations or anti-derivations, \emph{Filomat} \textbf{32}, no. 13, 4543-4554 (2018).

\bibitem{GarPeUnitCstaralg20} J. Garc{\'e}s, A.M. Peralta, One-parameter groups of orthogonality preservers on C$^*$-algebras, to appear in \emph{Banach J. Math. Anal.} arXiv:2004.04155

\bibitem{Gold} S. Goldstein, Stationarity of operator algebras, \emph{J. Funct. Anal.} \textbf{118}, no. 2, 275--308 (1993).
%
%\bibitem{GoldPas92} S. Goldstein, A. Paszkiewicz, Linear combinations of projections in von Neumann algebras, \emph{Proc. Amer. Math. Soc.} \textbf{116}, 175-183 (1992).

%\bibitem{Halpern69} H. Halpern, Commutators in  properly  infinite  von  Neumann  algebras, \emph{Trans. Amer. Math. Soc.} \textbf{139}, 55-73 (1969).
%
\bibitem{HOS} H. Hanche-Olsen, E. St{\o}rmer, \emph{Jordan Operator Algebras}, Pitman, London, 1984.

%\bibitem{HarMb92} R. Harte, M. Mbekhta, On generalized inverses in C$^*$-algebras, \emph{Studia Math.} \textbf{103}, no. 1, 71-77 (1992).

%\bibitem{HarMb93} R. Harte, M. Mbekhta, Generalized inverses in C$^*$*-algebras. II, \emph{Studia Math.} \textbf{106}, no. 2, 129-138  (1993).

\bibitem{Harris74} L.A. Harris, {Bounded symmetric homogeneous domains in infinite dimensional spaces}. In:
\emph{Proceedings on infinite dimensional Holomorphy (Kentucky 1973)}; pp. 13-40. Lecture Notes in
Math. 364. Berlin-Heidelberg-New York: Springer 1974.

\bibitem{HejNik96} S. Hejazian, A. Niknam, Modules Annihilators and module derivations of JB$^*$-algebras, \emph{Indian J. pure appl. Math.}, \textbf{27}, 129--140 (1996).

\bibitem{HoMarPeRu} T. Ho, J. Martinez-Moreno, A.M. Peralta, B. Russo, Derivations on real and complex JB$^\ast$-triples, \emph{J. London Math. Soc.} (2) \textbf{65}, no. 1, 85--102 (2002).

\bibitem{HoPeRu} T. Ho, A.M. Peralta, B. Russo, Ternary Weakly Amenable C$^*$-algebras and JB$^*$-triples, \emph{Quart. J. Math. (Oxford)} \textbf{64}, no. 4, 1109--1139 (2013).

\bibitem{JamPeSidd2015} F.B. Jamjoom, A.M. Peralta, A.A. Siddiqui, Jordan weak amenability and orthogonal forms on JB$^*$-algebras, \emph{Banach J. Math. Anal.} \textbf{9}, no. 4, 126--145 (2015).

\bibitem{JamPeSiddTah2015} F.B. Jamjoom, A.M. Peralta, A.A. Siddiqui, H.M. Tahlawi, Approximation and convex decomposition by extremals and the $\lambda$-function in JBW$^*$-triples, \emph{Quart. J. Math.} \textbf{66}, no. 2, 583--603 (2015).

%\bibitem{Jar90} K. Jarosz, Automatic continuity of separating linear isomorphisms, \emph{Canad. Math. Bull.} \textbf{33}, no. 2, 139--144 (1990).

%\bibitem{Jing2009} W. Jing, On Jordan all-derivable points of $B(H)$, \emph{Linear Algebra Appl.} \textbf{430} (4), 941-946 (2009).
%
%\bibitem{JingLuLi2002} W. Jing, S.J. Lu, P.T. Li, Characterization of derivations on some operator algebras, \emph{Bull. Austral. Math. Soc.} \textbf{66}, 227-232 (2002).

%\bibitem{John96} B.E. Johnson, Symmetric amenability and the nonexistence of Lie and Jordan derivations, \emph{Math. Proc. Cambridge Philos. Soc.} \textbf{120}, no. 3, 455--473 (1996).

%\bibitem{JohnKadRing72} B.E. Johnson, R.V. Kadison, J.R. Ringrose, Cohomology of operator algebras III. Reducting to normal cohomology, \emph{Bull. Sot. Math. France} \textbf{100}, 73-96 (1972).

%\bibitem{Kad51} R.V. Kadison, \textit{Isometries of operator algebras}, {Ann. of Math.} (2) {\bf 54}, 325--338 (1951).

\bibitem{Ka0} W. Kaup, Algebraic Characterization of symmetric complex Banach manifolds\hyphenation{manifolds}, \emph{Math. Ann.} \textbf{228}, 39--64 (1977).

\bibitem{Ka} W. Kaup, A Riemann Mapping Theorem for bounded symmentric domains in complex Banach spaces, \emph{Math. Z.} \textbf{183}, 503--529 (1983).

%\bibitem{kaup1984contractive} W. Kaup, Contractive projections on {J}ordan {$C^{\ast} $}-algebras and generalizations, {\em Math. Scand.} \textbf{54}, 1, 95--100 (1984).

\bibitem{Ka96} W. Kaup, On spectral and singular values in JB$^*$-triples, \emph{Proc. Roy. Irish Acad. Sect. A} \textbf{96}, no. 1, 95--103 (1996).

%\bibitem{KeLiWong04} W.-F. Ke, B.-R. Li, N.-C.Wong, Zero product preserving maps of continuous operator valued functions, \emph{Proc. Amer. Math. Soc.}, \textbf{132}, 1979-1985 (2004).

%\bibitem{LeuTsaiWong12} C.-W. Leung, C.-W. Tsai, N.-C. Wong, Linear disjointness preservers of W$^*$-algebras, \emph{Math. Z.} \textbf{270}, 699--708  (2012).

%\bibitem{LeuWong10} C.-W. Leung, N.-C. Wong, Zero product preserving linear maps of CCR C$^*$-algebras with Hausdorff spectrum, \emph{J. Math. Anal. Appl.} \textbf{361}, 187--194 (2010).

%\bibitem{LiPan} J. Li, Z. Pan, Annihilator-preserving maps, multipliers, and derivations, \emph{Linear Algebra Appl.} \textbf{423}, 5-13 (2010).

%\bibitem{LinMat} Y.-F. Lin and M. Mathieu, Jordan isomorphism of purely infinite C$\sp *$-algebras, \emph{Quart. J. Math.} \textbf{58}, 249-253 (2007).

%\bibitem{LiuChouLiaoWong2018} J.-H. Liu, C.-Y. Chou, C.-J. Liao, N.-C. Wong, Disjointness preservers of AW$^*$-algebras, \emph{Linear Algebra Appl.} \textbf{552}, 71--84 (2018).

%\bibitem{LiuChouLiaoWong2018b} J.-H. Liu, C.-Y. Chou, C.-J. Liao, N.-C. Wong, Linear disjointness preservers of operator algebras and related structures, \emph{Acta Sci. Math. (Szeged)} \textbf{84}, no. 1-2, 277--307 (2018).

%\bibitem{Lu2009} F. Lu, Characterizations of derivations and Jordan derivations on Banach algebras,\emph{Linear Algebra Appl.} \textbf{430}, No. 8-9, 2233-2239 (2009).

%\bibitem{Mar02} L.W. Marcoux, On the Linear Span of the Projections in Certain Simple C$^*$-algebras, \emph{Indiana Univ. Math. J.} \textbf{51}, 753-771 (2002).

%\bibitem{Marc2006} L.W. Marcoux, Sums of small numbers of commutators, \emph{J. Operator Theory} \textbf{56}, 111-142 (2006).
%
%\bibitem{Marc2010} L.W. Marcoux, Projections, commutators and Lie ideals in C$^*$-algebras, \emph{Math. Proc. R. Ir. Acad.} \textbf{110A}, 31-55 (2010).

%\bibitem{MarMur98} L. W. Marcoux, G.J. Murphy, Unitarily-invariant linear spaces in C$^*$-algebras, \emph{Proc. Amer. Math. Soc.} \textbf{126}, no. 12, 3597-3605 (1998).

%\bibitem{OikPe13} T. Oikhberg, A.M. Peralta, Automatic continuity of orthogonality preservers on a non-commutative $L_p(\tau)$-space, \emph{J. Funct. Anal.} \textbf{264}, no. 8, 1848--1872 (2013).

%\bibitem{PearTopp67} C. Pearcy, D. Topping, Sum of small numbers of idempotent, \emph{Michigan Math. J.} \textbf{14}, 453-465 (1967).

%\bibitem{Ped} G.K. Pedersen, \emph{C$^*$-algebras and their automorphism groups}, Academic Press, London 1979.

\bibitem{PedS88} S. Pedersen, Groups of isometries on operator algebras, \emph{Studia Math.} \textbf{90}, no. 2, 103--116 (1988).

%\bibitem{PedS88a} S. Pedersen, Groups of isometries on operator algebras. II, \emph{Rocky Mountain J. Math.} \textbf{18}, no. 4, 753-765 (1988).

\bibitem{PeRo2001} A.M. Peralta, A. Rodr\'{\i}guez Palacios, Grothendieck's inequalities for real and complex ${\rm JBW}\sp *$-triples, \emph{Proc. London Math. Soc.} (3) \textbf{83}, no. 3, 605--625 (2001).

%\bibitem{PeRu2014} A. M. Peralta, B. Russo, Automatic continuity of triple derivations on C$^*$-algebras and JB$^*$-triples, \emph{Journal of Algebra}  \textbf{399}, 960-977 (2014).

%\bibitem{Pop2002} C. Pop, Finite sums of commutators, \emph{Proc. Amer. Math. Soc.} \textbf{130}, 3039-3041 (2002).

%\bibitem{Ringrose72} J. R. Ringrose, {Automatic continuity of derivations of operator algebras}, \emph{J. London Math. Soc.} (2) \textbf{5} , 432-438 (1972).

\bibitem{RodPa91} A. Rodr{\'i}guez-Palacios, On the strong$^*$ topology of a JBW$^*$-triple, \emph{Quart. J. Math. Oxford Ser.} (2) \textbf{42}, no. 165, 99--103 (1991).

%\bibitem{Rud91} W. Rudin, \emph{Functional analysis}. Second edition. International Series in Pure and Applied Mathematics. McGraw-Hill, Inc., New York, 1991.

%\bibitem{Sak60} S. Sakai, On a conjecture of Kaplansky, \emph{Tohoku Math. J.}, \textbf{12}, 31-33 (1960).

%\bibitem{Sa} S. Sakai, \textit{C*- and W*-algebras}, Springer-Verlag, Berlin-New York, 1971.

%\bibitem{stacho1982projection} L.L. Stacho, A projection principle concerning biholomorphic automorphisms, {\em Acta Sci. Math.} \textbf{44}, 99--124 (1982).

%\bibitem{TagTav2019b} A. Taghavi, E. Tavakoli, Additivity of maps preserving Jordan triple products on prime C$^*$-algebras, preprint 2019.

%\bibitem{Tak} M. Takesaki, \newblock {\em Theory of operator algebras I}, \newblock Springer, New York, 2003.

\bibitem{Topping} D. Topping, Jordan algebras of self-adjoint operators, \emph{Mem. Amer. Math. Soc.} 53 (1965).

%\bibitem{Tsai11} C.-W. Tsai, The orthogonality structure determines a C$^*$-algebra with continuous
%trace, \emph{Operators and Matrices} \textbf{5}, 529-540 (2011).

%\bibitem{TsaiWong10} C.-W. Tsai, N.-C. Wong, Linear orthogonality preservers of standard operator algebras, \emph{Taiwanese J. Math.} \textbf{14}, 1047--1053 (2010).

\bibitem{Upmeier0} H. Upmeier, Derivations of Jordan C$^*$-algebras, \textit{Math. Scand.} \textbf{46}, 251--264 (1980).

\bibitem{Wolff94} M. Wolff, Disjointness preserving operators in C$^*$-algebras, \emph{Arch. Math.} \textbf{62}, 248--253 (1994).

%\bibitem{Wong2005}  N.-C. Wong, Triple homomorphisms of C$^*$-algebras, \emph{Southeast Asian Bulletin of Mathematics} \textbf{29}, 401--407 (2005).

%\bibitem{Wong2007} N.-C. Wong, Zero product preservers of C$^*$-algebras, \emph{Contemporary Math.} \textbf{435}, 377--380 (2007).

\bibitem{Wright77} J.D.M. Wright, Jordan C$^*$-algebras, \emph{Michigan Math. J.} \textbf{24}, 291--302 (1977).

\bibitem{Young78} M.A. Youngson, A Vidav theorem for Banach Jordan algebras, \emph{Math. Proc. Cambridge Philos. Soc.} \textbf{84}, no. 2, 263--272 (1978).

\end{thebibliography}


\medskip\medskip

\end{document}